\documentclass[11pt, twoside, leqno]{article}

\usepackage{amssymb,latexsym}
\usepackage{enumerate}
\usepackage{color}
\usepackage{amsmath, amsthm, amsfonts, amssymb, mathrsfs}
\newtheorem{theorem}{Theorem}[section]
\newtheorem{corollary}[theorem]{Corollary}
\newtheorem{lemma}[theorem]{Lemma}

\theoremstyle{definition}

\def\dint{\displaystyle\int}
\def\dsum{\displaystyle\sum}
\def\dsup{\displaystyle\sup}

\numberwithin{equation}{section}
\def\dfrac{\displaystyle\frac}

\numberwithin{equation}{section}

\numberwithin{equation}{section}

\begin{document}
\title{{\LARGE\bf     {Some jump and variational inequalities for the Calder\'{o}n commutators and related operators}
 \footnotetext{\small
{{\it MR(2000) Subject Classification}:\ } 42B20, 42B25}
\thanks {The research was supported by
NSF of China (Grant: 11471033, 11371057, 11571160, 11601396), NCET of China (Grant: NCET-11-0574), Thousand Youth Talents Plan  of China (Grant: 429900018-101150(2016)),  Funds for Talents of China (Grant: 413100002), the Fundamental Research Funds for the Central Universities (FRF-TP-12-006B, 2014KJJCA10) and SRFDP of China   (Grant: 20130003110003).}}}

\date{ }
\maketitle

\begin{center}
{\bf Yanping Chen \footnote{\small {Corresponding author.\ }}}\\
Department of Applied Mathematics,  School of Mathematics and Physics,\\
 University of Science and Technology Beijing,\\
Beijing 100083, The People's Republic of China \\
 E-mail: {\it yanpingch@ustb.edu.cn}
\vskip 0.1cm

{\bf Yong Ding}\\
School of Mathematical Sciences, Beijing Normal University,\\
Laboratory of Mathematics and Complex Systems (BNU),
Ministry of Education,\\
Beijing 100875,  The People's Republic of China \\
E-mail: {\it dingy@bnu.edu.cn}\vskip 0.1cm

{\bf Guixiang Hong}\\
School of Mathematics and Statistics, Wuhan University,\\
Wuhan 430072,  The People's Republic of China \\
E-mail: {\it guixiang.hong@whu.edu.cn}\vskip 0.1cm and

{\bf Jie Xiao}\\
Department of Mathematics and Statistics, Memorial University of Newfoundland,\\
St. Johns, Newfoundland and Labrador, Canada A1C 5S7\\
E-mail: {\it jxiao@mun.ca}

\end{center}

\begin{center}
\begin{minipage}{135mm}{\small
{\bf ABSTRACT\ }
In this paper, we establish jump and variational inequalities for  the Calder\'{o}n commutators, which are typical examples of non-convolution Calder\'on-Zygmund operators. For this purpose, we also show jump and variational inequalities for para-products and commutators from pseudo-differential calculus, which are of independent interest. New ingredients in the proofs involve identifying Carleson measures constructed from sequences of stopping times, in addition to many Littlewood-Paley type estimates with gradient.}
\end{minipage}
\end{center}
\vspace{0.3cm}

 \section{Introduction}\label{sect1}

Motivated by the modulus of continuity of Brownian motion, L\'epingle \cite{Lep76} established the first variational inequality for general martingales among many other interesting results. In \cite{PiXu88}, Pisier and Xu established implicitly the jump inequality (explicitly stated in Lemma 6.7 of \cite{JKRW98}), and then by real interpolation provided another proof of L\'epingle's variational inequality. The advantage of Pisier and Xu's approach is that it works also for vector-valued martingales.

\bigskip

Bourgain \cite{Bou89} is the first one who exploited L\'epingle's result to obtain corresponding variational estimates for the Birkhoff ergodic averages along subsequences of natural numbers and then directly deduce point-wise convergence results without previous knowledge that point-wise convergence holds for a dense subclass of functions, which are not available in some ergodic models. In particular, Bourgain's work \cite{Bou89} has initiated a new research direction in ergodic theory and harmonic
analysis. In \cite{JKRW98, JRW03, JRW00, CJRW2000, CJRW2002}, Jones and his collaborators systematically studied variational inequalities for ergodic averages and truncated singular integrals of homogeneous type. Since then many other publications came to enrich the literature on this subject (cf. e.g. \cite{GiTo04, LeXu2, DMT12, JSW08, MaTo, OSTTW12, Hon1}). Recently, several works on weighted as well as vector-valued variational inequalities in ergodic theory and harmonic analysis have also appeared (cf. e.g. \cite{MTX, KZ, HLP}); and several results on $\ell^p(\mathbb Z^d)$-estimates of $q$-variations for discrete operators of Radon type have also been established (cf. e.g. \cite{Kra14, MiTr14, MST15, MTZ14, Zor14}).

\bigskip

All the operators considered in the previous cited papers have nice symmetry properties, for instance, semigroup property or dilation invariance properties. So far as we know, it is still unknown whether jump and variational inequalities hold for all singular integrals of convolution type (see \cite{MST15}), let alone for all standard Calder\'on-Zygmund operators. However, in this paper, we manage to show jump and variational inequalities for the Calder\'{o}n commutators, which are typical examples of non-convolution Calder\'on-Zygmund operators. For this purpose, we first show jump and variational inequalities for para-products and commutators from pseudo-differential calculus, which are of independent interest.

\bigskip

The Calder\'{o}n commutators (see \cite{C1, C2}) originate from a representation of linear differential operators by means of singular integral operators, which is an approach to the uniqueness of the Cauchy problem for partial differential equations (see \cite{C3}).
Given a positive integer $m$, every linear partial differential operator $L$ of homogeneous order $m$ with bounded variable coefficients on Euclidean space $\mathbb R^n$ can be expressed as
$$ Lf=T\Lambda^mf,$$ where $\widehat{\Lambda f}=\varphi(\xi)\widehat{f}(\xi)$, $\varphi (\xi)$ is a positive infinitely differentiable
function such that $\varphi(\xi)=|\xi|$ if $|\xi|\ge 1,$
and $T$ is a singular integral operator
$$
Tf =\dint  |\xi|^{-m}\dsum_{|\gamma|=m}b_\gamma(x)(-i\xi)^\gamma e^{ix\cdot \xi}\widehat{f}(\xi) d\xi +\int r(x,\xi)e^{ix\cdot \xi}\widehat{f}(\xi) d\xi$$
with $\gamma$ being an multi-indices of non-negative integers and $|\gamma|=\gamma_1+\dotsm+\gamma_n$.
Let
$B$ be the operator given by the multiplication of the
Lipschitzian function $ b(x)$. For simplicity, let us consider the case $n=1$, let $H$ be the Hilbert transform, as it is well known,
this transform can be expressed as follows
$$ Hf(x)= -\frac{i}{2\pi}\int_{-\infty}^\infty sgn \xi e^{ix\cdot\xi}\widehat{f}(\xi)\,d\xi$$
and this makes it clear that $ B, H$ and $B H$ are operators of the type of the generalized
$T$ and the simplest of their kind. In order to show that  $ HB $ is of the same
type, since $HB=B H-[b, H],$ it would suffice to show that $[b, H]\frac{d}{dx}$ is bounded in $L^p(\Bbb R),\,1<p<\infty.$  Calder\'{o}n [\ref{C2}] introduced the first Calder\'{o}n commutator which is defined by
$$ [b, H\frac{d}{dx}]f(x):=\text{p.v.}\int_{-\infty}^\infty\frac{(-1)}{x-y}\frac{(b(x)-b(y))}{x-y}f (y)\,dy.$$The integral on the right, which in the case $b(x)=x$ reduces to the Hilbert transform,
is the one studied in [\ref{C1}].
Note that
$$[b, H]\frac{d}{dx}=[b, H\frac{d}{dx}]-H[b,\frac{d}{dx}]$$
and since the operator $[b,\frac{d}{dx}]$ is multiplication by $b'(x),$ which is a bounded function if $b(x)$ is Lipschitizian, $H[b,\frac{d}{dx}]$ is bounded in $L^p(\Bbb R)$ and the continuity of $[b, H]\frac{d}{dx}$ is equivalent to that of $[b, H\frac{d}{dx}].$ Thus, the role of the first Calder\'on commutator
in the theory of partial differential equations becomes apparent.
 Commutator $[b, H\frac{d}{dx}]$ also plays an important role in the theory of  Cauchy integral along Lipschitz curve in $\Bbb C$ and the Kato square root problem on $\Bbb R$ (see \cite{C3, Fe, Me, MC} for the details).

\bigskip


As Calder\'on did in \cite{C1}, there are large classes of commutators which are of independent interest in harmonic analysis. For $\varepsilon>0,$ suppose that  $\mathcal{C}_\varepsilon f$ is the truncated Calder\'{o}n commutator which is defined by
 \begin{equation}\label{Cal}
\mathcal{C}_\varepsilon f
(x)=\int_{|x-y|>\varepsilon}\cfrac{\Omega(x-y)}{|x-y|^{n+1}}\ (b(x)-b(y))f(y)dy,
 \end{equation}
  where  $\Omega$ is
homogeneous of degree zero, integrable on $\mathbf S^{n-1}$ (
the unit sphere in $\mathbb R^n$) and satisfies
\begin{equation}\label{mean zero}\int_{\mathbf S^{n-1}}\Omega(x')(x'_k)^N\,d\sigma(x')=0,\, k=1,\dots, n,\end{equation} for all integers $0\le N\le 1$. Then for $f\in C_0^\infty(\mathbb R^n),$ the  Calder\'{o}n commutator
\begin{align}\label{point pv of C}
\mathcal{C}f(x)=\lim_{\varepsilon\rightarrow0^+}\mathcal{C}_\varepsilon f(x), \ a.e. \ x\in\mathbb R^n.
\end{align}
Using the method of rotation, Calder\'{o}n [\ref{C1}] has shown the boundedness of the commutator $\mathcal{C}$ and as a consequence obtained the boundedness of the operators $[b,T]\nabla$ and $\nabla[b,T]$, where $T$ is a homogeneous singular integral operator with some symbol $K$ which can be defined similarly as $\mathcal{C}$: For $f\in C_0^\infty(\mathbb R^n)$
\begin{equation}\label{SIO}
Tf(x)=\lim_{\varepsilon\rightarrow0^+}T_\varepsilon f(x), \ a.e. \ x\in\mathbb R^n
\end{equation}
with
 \begin{equation}\label{tr of S}
T_\varepsilon
f(x)=\int_{|y|>\varepsilon}K(y)f(x-y)dy,
 \end{equation}
where  $K$ is homogeneous of degree $-n$, belongs to  $ L_{\rm loc}^1({\Bbb R}^n)$ and
 satisfies the cancelation condition
\begin{equation}\label{can of O}
\int_{\mathbf S^{n-1}}K(y')d\sigma(y')=0.
 \end{equation}
Later on, many authors made important progress on the Calder\'on commutators, one can consult \cite{Co, CM, CM1, Mu, SH, SH1, Y, Ta3, MMT, MuWh71, GH, CDH1}, among
numerous references, for its development and applications.

\bigskip

 That the point-wise principle value (\ref{point pv of C}) exists for all $f\in L^p(\mathbb R^n)$ follows from the maximal inequality which was established in \cite{MuWh71}. In the present paper, the variational inequality that we will show implies the maximal inequality due to \eqref{number contr ineq} below. Moreover our result provide quantitative information of the convergence.

\bigskip

In order to present our results in a precise way, let us fix some notations.
 Given a family of complex numbers $\mathfrak{a}=\{a_t: t\in\mathbb{R}_+\}$ and $0<q<\infty$, the strong and the weak $q$-variation norm of the family $\mathfrak{a}$ is defined respectively by
\begin{equation}\label{q-ver number family}
\|\mathfrak{a}\|_{V_q}=\sup\|(a_{t_k}-a_{t_{k-1}})_{k\geq1}\|_{\ell^{q}},
\end{equation}
and
\begin{equation}\label{weak q-ver number family}
\|\mathfrak{a}\|_{V_{q,\infty}}=\sup\|(a_{t_k}-a_{t_{k-1}})_{k\geq1}\|_{\ell^{q,\infty}}
\end{equation}
where the supremum runs over all increasing sequences $\{t_k:k\geq0\}$. Here $\ell^q$ (resp. $\ell^{q,\infty}$) denote the Lebesgue $L^q$ (resp. weak $L^q$) norm on the set of integers.
From the definition, it is quite clear that the following inequalities hold: For any $0<r<q<\infty$,
\begin{align}\label{embed}
\|\mathfrak{a}\|_{V_{q,\infty}}\leq \|\mathfrak{a}\|_{V_{q}}\leq \|\mathfrak{a}\|_{V_{r,\infty}}.
\end{align}
On the other hand, it is also trivial that
\begin{equation}\label{number contr ineq}
\|\mathfrak{a}\|_{L^\infty}:=\sup_{t\in\mathbb R_+}|a_t|
\le\|\mathfrak{a}\|_{V_q}+|a_{t_0}|\quad\text{for}\ \ 0<q<\infty,
\end{equation}
for some fixed $t_0$.

\bigskip

Via the definition of the strong and weak $q$-variation norm of a family of numbers, one may define the strong and the weak $q$-variation function $V_q(\mathcal F)$ and $V_{q,\infty}(\mathcal F)$ of a family $\mathcal F$ of functions.
Given a family of Lebesgue measurable functions $\mathcal F=\{F_t:t\in\mathbb{R}_+\}$ defined on $\mathbb{R}^n$, for fixed $x$ in $\mathbb{R}^n$, the value of the strong $q$-variation function $V_q(\mathcal F)$ of the family $\mathcal F$ at $x$ is defined by
\begin{equation}\label{defini of Vq(F)}
V_q(\mathcal F)(x)=\|\{F_t(x)\}_{t\in\mathbb{R}}\|_{V_q},\quad q>0;
\end{equation}
while the value of the weak $q$-variation function $V_{q,\infty}(\mathcal F)$ of the family $\mathcal F$ at $x$ is defined by
\begin{equation}\label{defini of Jq(F)}
V_{q,\infty}(\mathcal F)(x)=\|\{F_t(x)\}_{t\in\mathbb{R}}\|_{V_{q,\infty}},\quad q>0;
\end{equation}

\bigskip

 Suppose $\mathcal{A}=\{{A}_t\}_{t>0}$ is a family of operators  on $L^p(\Bbb R^n)\, (1\le p\le\infty)$. The related strong and weak $q$-variation operator are simply defined respectively as
$$V_q(\mathcal Af)(x)=\|\{A_t(f)(x)\}_{t>0}\|_{V_q},\quad\forall f\in L^p(\mathbb{R}^n)$$
and
$$V_{q,\infty}(\mathcal Af)(x)=\|\{A_t(f)(x)\}_{t>0}\|_{V_{q,\infty}},\quad\forall f\in L^p(\mathbb{R}^n). $$
It is easy to observe from the definition of $q$-variation norm that for any $x$ if $V_{q,\infty}(\mathcal Af)(x)<\infty$ for some $q<\infty$, then $\{A_t(f)(x)\}_{t>0}$ converges when $t\rightarrow0$ or $t\rightarrow\infty$. In particular, if $V_{q,\infty}(\mathcal Af)$ belongs to some function spaces such as $L^p(\mathbb{R}^n)$ or $L^{p,\infty}(\mathbb{R}^n)$, then the sequence converges almost everywhere without any additional condition. This is why mapping property of strong or weak $q$-variation operator is so interesting in ergodic theory and harmonic analysis. On the other hand, from \eqref{number contr ineq}, variational inequality is much stronger than corresponding maximal inequality. Namely, for any $f\in L^p(\Bbb R^n)$ and $x\in\Bbb R^n$, we have
 \begin{equation}\label{control of maxi opera}
A^\ast(f)(x)\le V_{q}(\mathcal{A}f)(x)\quad\text{for}\ \ q\ge1,
\end{equation}
where $A^\ast$ is the maximal operator defined by
$$A^\ast(f)(x):=\sup_{t>0}|{A}_t(f)(x)|$$ and thus is more interesting.

\bigskip

As we know for a family of Lebesgue measurable functions $\mathcal F=\{F_t(x):t\in\mathbb{R}_+\}$, there is another related notion called $\lambda$-jump function $N_\lambda(\mathcal F)$ whose value at $x$ is defined as the supremum over all $N$ such that there exist $t_0<t_1<t_2<\dotsc<t_N$ with
$$|F_{t_k}(x)-F_{t_{k-1}}(x)|>\lambda$$
for all $k=1,\dotsc, N$. It is easy to check that this function is related with the weak $q$-variation norm as follows
$$V_{q,\infty}(\mathcal F)(x)=\sup_{\lambda>0}\lambda (N_\lambda(\mathcal F)(x))^{1/q}.$$ We refer the reader to \cite{JSW08} for more information on $\lambda$-jump functions.

\bigskip

Now, we can formulate our main result as follows.

\bigskip

\begin{theorem}\label{thm:H}  Let $b\in Lip(\mathbb R^n)$ and $\mathscr C =\{\mathcal C_{\varepsilon} \}_{\varepsilon>0}$ where $\mathcal C_{\varepsilon} (f)$ are as in \eqref{Cal}  with $\Omega$ satisfying \eqref{mean zero}.
 If $\Omega\in L(\log^+\!\! L)^2(\mathbf S^{n-1})$, then the following jump inequality holds for $1< p<\infty,$ namely,
\begin{equation}\label{N12w}
\sup_{\lambda>0}\|\lambda\sqrt{N_\lambda(\mathscr C f)}\|_{L^p}\le
C_{p,n,\Omega}\|\nabla b\|_{L^\infty}\|f\|_{L^p}.\end{equation}
\end{theorem}

\bigskip

Quite remarkably, on the one hand, as in \cite{PiXu88, Bou89, JSW08}, the jump inequality \eqref{N12w} implies all the strong $q$-variational inequality ($2<q<\infty$) by a real interpolation argument (see for instance Lemma 2.1 of \cite{JSW08}), and thus implies all the weak $q$-variational inequality and maximal inequality by \eqref{embed} and \eqref{number contr ineq}. For this reason, we will not state explicitly any $q$-variational inequality in the present paper. On the other hand, the strong $2$-variational inequality, and thus any $q$-variational inequality ($q<2$) may fail, see \cite{Qia98} and \cite{AJS98} for related information. However, it is still unknown whether the weak $2$-variational inequality holds, that is, whether the estimate \eqref{N12w} is still true if the supremum over $\lambda$ can be put inside the norm.

\bigskip

Consequently, we have the following Calder\'{o}n-type estimates.

\begin{corollary}\label{thm:H2}  Let $1<p<\infty$. Let $b\in Lip(\mathbb R^n)$ and $\mathcal T_b =\{[b,T_{\varepsilon}]\}_{\varepsilon>0}$ where $T_{\varepsilon}$ are as in \eqref{tr of S}.
Suppose that  $K(x)$ have locally integrable first-order derivatives in $|x|>0$
and suppose that $K(x)$ and the partial derivatives  of $K(x)$ belong locally to $L(\log^+\!\!L)^2$ in $|x|>0$. If $f$ is continuously differentiable and have compact support, then the following jump inequality holds  namely,
$$
\sup_{\lambda>0}\|\lambda\sqrt{N_\lambda(\mathcal T_b (\nabla f))}\|_{L^p}\le
C_{p,n,K}\|\nabla b\|_{L^\infty}\|f\|_{L^p}.$$ Furthermore, if for $\varepsilon>0,$ $[b,T_{\varepsilon}]f $ has first-order derivatives
in $L^p({\Bbb R}^n)$, then $$\sup_{\lambda>0}\|\lambda\sqrt{N_\lambda(\nabla \mathcal T_b  f )}\|_{L^p}\le
C_{p,n,K}\|\nabla b\|_{L^\infty}\|f\|_{L^p}.$$
\end{corollary}

As in most of the previously cited paper (in particular see Lemma 1.3 in \cite{JSW08}), we shall show estimate \eqref{N12w} by showing separately the corresponding inequalities for the long and short variation. That is, we are reduced to prove for $1<p<\infty$
\begin{align}\label{dyadic jump singular}
\sup_{\lambda>0}\|\lambda\sqrt{N_\lambda(\{\mathcal C_{2^k}f\}_k)}\|_{L^p}\le
C_{p,n,\Omega}\|\nabla b\|_{L^\infty}\|f\|_{L^p}
\end{align}
and
\begin{align}\label{short variation singular}
\|S_2(\mathscr Cf)\|_{ L^p}\le
 C_{p,n,\Omega}\|\nabla b\|_{L^\infty}\|f\|_{L^p},
\end{align}
where
 $$
S_2(\mathscr Cf)(x)=\bigg(\sum_{j\in\mathbb Z}[V_{2,j}(\mathscr Cf)(x)]^2\bigg)^{1/2},
$$
with
 $$
V_{2,j}(\mathscr Cf)(x)=\bigg(\sup_
{2^j\leq t_0<\cdots<t_N<2^{j+1}}{\sum_{k=0}^{N-1}}|\mathcal C_{t_{k+1}}f(x)-\mathcal C_{t_k}f(x)|^2\bigg)^{1/2}.
 $$

Although we encounter some difficulties in proving \eqref{short variation singular} (see for instance some lemmas in Section 5), the novelty of the proof lies in showing \eqref{dyadic jump singular}. We need two results which are of independent interest.

\bigskip

The first auxiliary result is variational inequality for some kind of para-product (see \cite{DMT12} for related results), whose formulation is motivated by the results on maximal operators of Duoandikoetcea and Rubio de Francia {\cite{DR86}} and para-products \cite{GR}. Let $\mu$ be a compactly supported finite Borel measure on ${\Bbb R}^n,$ that is, $\mu$ is absolutely continuous on
the Lebesgue measure $dx$, its Radon-Nikod\`{y}m derivative is a nonnegative Lebesgue measurabe
function on ${\Bbb R}^n$ with compact support set. We consider dilates $\mu_k$  of $\mu$  defined with respect to a group of dilations $\{2^k\}_{k\in \Bbb Z}$ defined by \begin{equation*} \int f(x)\,d\mu_k(x):=\int_{{\Bbb R}^n}f(2^kx)\,d\mu.\end{equation*}  Let $\Upsilon_kf(x)=\mu_k\ast f(x)$ for $k\in \Bbb Z.$ A well known fact is, if $\mu$ satisfies the Fourier transform:
\begin{align}\label{condition}|\widehat{\mu}(\xi)|\le C|\xi|^{-\alpha},\end{align} for some $\alpha>0,$ then  the maximal operator defined as $M_\mu f(x)=\sup_{k\in \Bbb Z}|\Upsilon_kf(x)|$ is bounded on $L^p({\Bbb R}^n),\,1<p<\infty$. Further, if the Radon-Nikod\`{y}m derivative $\zeta$ of $\mu$ satisfies the stronger condition:\begin{align}\label{condition1}\int_{{\Bbb R}^n}|\zeta(x+y)-\zeta(x)|\,dx\le C|y|^\tau\end{align} for some $\tau>0$, then  $M_\mu$  is of weak type $(1,1)$(see \cite{S1993}). Clearly \eqref{condition1} implies \eqref{condition}.  Suppose that $\phi(x)\in \mathscr{S}({\Bbb R}^n)$. Denote  $\Phi_k f(x)=\phi_k\ast f(x),$  where $\phi_k(x)=2^{-kn}\phi(2^{-k}x).$  Give a function $b$ on ${\Bbb R}^n$, we define the operator as follows: \begin{equation}\label{ub} \mathscr U_b f=\{ (\Phi_kb)(\Upsilon_kf)\}_{k\in \Bbb Z},\end{equation} for $f\in L_{\rm loc}^1({\Bbb R}^n).$
 They are not operators of convolution type. We note that the transpose of the family of operators $\mathscr U_b$ is formally given by the identity
\begin{equation*} \mathscr U_b^t f=\{ \Upsilon_k(f\Phi_kb)\}_{k\in \Bbb Z}.\end{equation*}
We are now ready to state the first auxiliary result as follows.

\begin{theorem}\label{Phi0} Suppose that $b\in L^\infty({\Bbb R}^n).$ Let $\mathscr U_b $ be defined as in \eqref{ub} with $\mu$ being a compactly supported finite Borel measure on ${\Bbb R}^n$ and  $\phi(x)\in \mathscr{S}({\Bbb R}^n)$.

{\rm (i)} If $\mu$ satisfies \eqref{condition}, then for $1<p<\infty$ and $f\in L^p({\Bbb R}^n)$,
we have$$\sup_{\lambda>0}\|\lambda\sqrt{N_\lambda(\mathscr U_b  f)}\|_{L^p}\le C_{p,n}\| b\|_{L^\infty}\|f\|_{L^p}.$$

{\rm (ii)}  In addition for $p=1$, if $\mu$ satisfies \eqref{condition1}, we have $$\sup_{\lambda>0}\|\lambda\sqrt{N_\lambda(\mathscr U_b  f)}\|_{L^{1,\infty}}\le C_{p,n}\| b\|_{L^\infty}\|f\|_{L^1}.$$
\end{theorem}

The proof of Theorem \ref{Phi0} involves identifying two Carleson measures constructed from sequences of conditional expectations, one of which is in turn constructed from sequences of stopping times, see below Lemma \ref{Phib}, \ref{Phib1} and \ref{lem:finding stop time}.

\bigskip

The second  auxiliary result we need is the variational inequalities for commutators  of pseudo differential calculus with Lipschitz functions.   Since the 1960s, the theory of pseudo differential operators has played an important role in many exciting and deep investigations into linear PDE (see \cite{Fe1,  CM, B, Ho, St, Ta0, Ta2, AT, Me, MC}).

\begin{theorem}\label{Phi1} For $k\in \Bbb Z,$ let $\Phi_k$ be defined as Theorem \ref{Phi0}. For  $b\in Lip({\Bbb R}^n),$ set $\mathscr F_b =\{[b, \Phi_k]  \}_{k}$. Suppose that $f$ is continuously differentiable and has compact support.

{\rm (i)} Then for $1<p<\infty,$ we have $$\sup_{\lambda>0}\|\lambda\sqrt{N_\lambda(\mathscr F_b (\nabla f))}\|_{L^p}\le C_{p,n}\|\nabla b\|_{L^\infty}\|f\|_{L^p}$$ and
$$\sup_{\lambda>0}\|\lambda\sqrt{N_\lambda( \nabla \mathscr F_bf)}\|_{L^p}\le C_{p,n}\|\nabla b\|_{L^\infty}\|f\|_{L^p}.$$

{\rm (ii)}  In addition for $p=1$ we have $$\sup_{\lambda>0}\|\lambda\sqrt{N_\lambda(\mathscr F_b (\nabla f))}\|_{L^{1,\infty}}\le C_{p,n}\|\nabla b\|_{L^\infty}\|f\|_{L^1}$$ and $$\sup_{\lambda>0}\|\lambda\sqrt{N_\lambda( \nabla \mathscr F_bf)}\|_{L^{1,\infty}}\le C_{p,n}\|\nabla b\|_{L^\infty}\|f\|_{L^1}.$$
\end{theorem}

The paper is organized as follows. In Section 2,   some key lemmas  will be introduced for the proof of  \ref{Phi0}.
In Section 3 and Section 4,  we give the proof of Theorem  \ref{Phi0} and  Theorem  \ref{Phi1}, respectively.  In Section 5, we give some lemmas for the proof of Theorem \ref{thm:H}. Section 6 and   Section 7 are devoted to  the proof of  Theorem \ref{thm:H}. In Section 8, we give the proof of Corollary \ref{thm:H2}.
For $p\ge 1,$ $p'$ denotes the conjugate
exponent of $p$, that is, $p'=p/(p-1).$  Throughout this paper, the letter $``C\,"$ will stand for a positive
constant which is independent of the essential variables and not
necessarily the same one in each occurrence.

\section{Some key lemmas}

Let us begin with some lemmas and their proofs, which will play a key role in proving Theorem \ref{Phi0}.
We borrow some notations and results from \cite[pp.6724]{JSW08}. For $j\in\mathbb Z$ and $\beta=(m_1,\cdots,m_n)\in\mathbb Z^n$,  we denote the dyadic cube $\prod_{k=1}^n(m_k2^j,(m_k+1)2^j]$ in $\mathbb R^n$ by $Q_\beta^j$, and the set of all dyadic cubes with side-length $2^j$ by $\mathcal D_j$. The conditional expectation of a local integrable $f$ with respect to $\mathcal D_j$ is given by
$$
\mathbb E_jf(x)=\sum_{Q\in \mathcal D_j}\frac1{|Q|}\int_{Q}f(y)dy\cdot\chi_{Q}(x)
$$
for all $j\in\mathbb Z$.

\begin{lemma}\label{Phib} Let $\phi\in {\mathscr
S}({\Bbb R}^n)$ and $\widehat{\phi}(0)=1.$ For $k\in \Bbb Z,$ denote by $\Phi_k f(x)=\phi_k\ast f(x),$ where $\phi_k(x)=2^{-kn}\phi(2^{-k}x).$  Let  $\mathbb E_k$ be given above and  $b\in BMO({\Bbb R}^n).$ Let $\delta_{2^{k}}(t)$ be Dirac mass at the point $t=2^{k}.$ Then there is a constant $C>0$ such that $$d\nu(x,t)=\dsum_{k\in \Bbb Z}|\Phi_kb(x)-\mathbb E_k b(x)|^2\,dx\,\delta_{2^k}(t)$$ is a Carleson measure on $\Bbb R_{+}^{n+1}$ with norm at most $C\|b\|_{\ast}^2.$
 \end{lemma}
{\emph{Proof.}} For a cube $Q$ in ${\Bbb R}^n$ we let $Q^*$ be the cube with the same center and orientation whose side length is $100\sqrt{n}\ell(Q)$, where $\ell(Q)$ is the side length of $Q.$ Fix a cube $Q$ in ${\Bbb R}^n,$ split $b$ as
$$b=(b-b_Q)\chi_{Q^*}+(b-b_Q)\chi_{(Q^*)^c}+b_Q.$$ Let $T(Q)=Q\times(0,\ell(Q)).$ Since $\Phi_kb_Q=b_Q$ and $\mathbb E_k b_Q=b_Q,$ then $$\Phi_kb_Q-\mathbb E_k b_Q=0.$$
Thus,
$$\begin{array}{cl}\nu(T(Q))=\dsum_{2^{k}\le \ell(Q)}\dint_{Q}|\Phi_k(b)(x)-\mathbb E_k(b)(x)|^2\,dx\le 2\Sigma_1+2\Sigma_2,\end{array}$$ where $$\Sigma_1=\dsum_{k\in \Bbb Z}\dint_{{\Bbb R}^n}|\Phi_k((b-b_Q)\chi_{Q^*})(x)-\mathbb E_k((b-b_Q)\chi_{Q^*})(x)|^2\,dx$$ and $$\Sigma_2=\dsum_{2^{k}\le \ell(Q)}\dint_{Q}\Phi_k((b-b_Q)\chi_{(Q^*)^c})(x)-\mathbb E_k((b-b_Q)\chi_{(Q^*)^c})(x)|^2\,dx.$$
Then\begin{align}\Sigma_1&\le C\dint_{Q^*}|b(x)-b_Q|^2\,dx\le C|Q|\| b\|_{\ast}^2 ,\end{align} where in the first inequality we have used that $$\bigg\|\bigg(\dsum_{k\in \Bbb Z}|\Phi_k(g)-\mathbb E_k(g)|^2\bigg)^{1/2}\bigg\|_{L^2}\le C\|g\|_{L^2}$$(see \cite{JSW08}).
For $\Sigma_2,$ we have
$$\begin{array}{cl}\Sigma_2&=\dsum_{2^{k}\le \ell(Q)}\dint_{Q}|\Phi_k((b-b_Q)\chi_{(Q^*)^c})(x)-\mathbb E_k((b-b_Q)\chi_{(Q^*)^c})(x)|^2\,dx\\&\le C\dsum_{2^{k}\le \ell(Q)}\dint_{Q}|\Phi_k((b-b_Q)\chi_{(Q^*)^c})(x)|^2\,dx+C\dsum_{2^{k}\le \ell(Q)}\dint_{Q}|\mathbb E_k((b-b_Q)\chi_{(Q^*)^c})(x)|^2\,dx.\end{array}$$
Since $\phi(x)\le \frac{1}{(1+|x|)^{n+\delta}}$ for some $\delta>1$,  then by the same argument of \cite{GR, GR1}, we get $$\begin{array}{cl}\dsum_{2^{k}\le \ell(Q)}\dint_{Q}|\Phi_k((b-b_Q)\chi_{(Q^*)^c})(x)|^2\,dx\le C|Q|\| b\|_{\ast}^2.\end{array}$$
Recall that
$$\begin{array}{cl}
\mathbb E_kf(x)=\dsum_{{\widetilde{Q}}\in \mathcal D_k}\dfrac1{|\widetilde{Q}|}\dint_{\widetilde{Q}}f(y)dy\cdot\chi_{\widetilde{Q}}(x).
\end{array}$$Then we get
$$\begin{array}{cl}\dsum_{2^{k}\le \ell(Q)}\dint_{Q}|\mathbb E_k((b-b_Q)\chi_{(Q^*)^c})(x)|^2\,dx&\le\dsum_{2^{k}\le \ell(Q)}\dsum_{{\widetilde{Q}}\in \mathcal D_k}\dint_{Q}|\frac1{|\widetilde{Q}|}\int_{\widetilde{Q}}(b-b_Q)\chi_{(Q^*)^c}(y)dy|^2\cdot\chi_{\widetilde{Q}}(x)\,dx\\&=\dsum_{2^{k}\le \ell(Q)}\dsum_{{\widetilde{Q}}\in \mathcal D_k}\dint_{Q\cap\widetilde{Q}}|\frac1{|\widetilde{Q}|}\int_{\widetilde{Q}\cap(Q^*)^c}(b(y)-b_Q)dy|^2\,dx.\end{array}$$
 If $\widetilde{Q}\cap {(Q^*)}^c\neq \varnothing$, since $\ell(\widetilde{Q})=2^k$ and $\ell(Q^*)= 100\sqrt{n}\ell(Q)\ge 100\sqrt{n} 2^k=100\sqrt{n}\ell(\widetilde{Q}),$ then we get $$\widetilde{Q}\cap Q=\varnothing.$$ Therefore, either $\widetilde{Q}\cap {(Q^*)}^c=\varnothing$ or $\widetilde{Q}\cap {(Q^*)}^c\neq \varnothing$, we can get
\begin{align}\label{E}\dsum_{2^{k}\le \ell(Q)}\dint_{Q}|\mathbb E_k((b-b_Q)\chi_{(Q^*)^c})(x)|^2\,dx=0.\end{align}Together,$$\begin{array}{cl}\Sigma_2\le C|Q|\| b\|_{\ast}^2.\end{array}$$ Then combined this with $\Sigma_1,$ we get $$\nu(T(Q))\le C|Q|\| b\|_{\ast}^2.$$ This says that $$d\nu(x,t)=\dsum_{k\in \Bbb Z}|\Phi_kb(x)-\mathbb E_k b(x)|^2\,dx\,\delta_{2^k}(t)$$ is a Carleson measure on $\Bbb R_{+}^{n+1}$ with norm at most $C\|b\|_{\ast}^2.$\qed

\begin{lemma}\label{Phib1} For $j\in \Bbb Z,$ let $\mathbb E_j$ be given above and  $b\in BMO({\Bbb R}^n).$ Let $\delta_{t_{k}}(t)$ be Dirac mass at the point $t=t_{k}.$ Then there is a constant $C>0$ such that $$d\nu(x,t)=\dsum_{k\ge 0}|\mathbb E_{t_{k+1}}b(x)-\mathbb E_{t_{k}} b(x)|^2\,dx\,\delta_{t_{k}}(t)$$ is a Carleson measure on $\Bbb R_{+}^{n+1}$ with norm at most $C\|b\|_{\ast}^2,$ where $\{t_k\}_{k\geq0}$ is any sequence of decreasing stopping times and the bound does not depend on the stopping
times.
 \end{lemma}
{\emph{Proof.}}
The proof is essentially similar to  Lemma \ref{Phib}. More precisely,   we  need to estimate in $\Sigma_1$ with $\Phi_k((b-b_Q)\chi_{Q^*})(x)-\mathbb E_k((b-b_Q)\chi_{Q^*})(x)$ replaced by $\mathbb E_{t_{k+1}}((b-b_Q)\chi_{Q^*})(x)-\mathbb E_{t_{k}}((b-b_Q)\chi_{Q^*})(x)$. The desired result follows from $$\Big\|(\sum_{k\ge 0}|\mathbb E_{t_{k+1}}(g)-\mathbb E_{t_{k}}(g)|^2)^{1/2}\Big\|_{L^2}\le C\|g\|_{L^2}$$
due to Burkholder-Gundy inequality since $\{\mathbb E_{t_{k}}(g)\}_{k\geq0}$ forms a new martingale (see for instance \cite{PiXu88}). In $\Sigma_2$, we replace $\Phi_k((b-b_Q)\chi_{(Q^*)^c})(x)-\mathbb E_k((b-b_Q)\chi_{(Q^*)^c})(x)$ with $\mathbb E_{t_{k+1}}((b-b_Q)\chi_{(Q^*)^c})(x)-\mathbb E_{t_{k}}((b-b_Q)\chi_{(Q^*)^c})(x)$. Even though stopping times are maps from $\mathbb R^n$ to integers, we can still use the same arguments used in proving \eqref{E}, and conclude that
$$\begin{array}{cl}\dint_{Q}\dsum_{2^{t_{k}}\le \ell(Q)}|\mathbb E_{t_{k+1}}((b-b_Q)\chi_{(Q^*)^c})(x)|^2\,dx=0\end{array}$$  and $$\begin{array}{cl}\dint_{Q}\dsum_{2^{t_{k}}\le \ell(Q)}|\mathbb E_{t_{k}}((b-b_Q)\chi_{(Q^*)^c})(x)|^2\,dx=0.\end{array}$$
Let us explain briefly the second identity. The first identity follows similarly. We first write the left hand side as
$$\dsum_{k\ge 0}\dint_{Q}\chi_{2^{t_{k}}\le \ell(Q)}|\mathbb E_{t_{k}}((b-b_Q)\chi_{(Q^*)^c})(x)|^2\,dx.$$
Then we claim the integrand equal zero. Indeed, for any $x\in Q$,
$$\mathbb E_{t_{k}}((b-b_Q)\chi_{(Q^*)^c})(x)=\frac{1}{|Q(t_{k}(x))|}\dint_{Q(t_{k}(x))\cap(Q^*)^c} (b(y)-b_Q)\,dy$$
where $Q(t_{k}(x))$ is the unique dyadic cube containing $x$ with side-length equal to $2^{t_{k}(x)}$. Then $\ell(Q(t_{k}(x)))\leq \ell(Q)$ implies $Q(t_{k}(x))\cap(Q^*)^c=\varnothing.$

\qed

\section{Proof of Theorem \ref{Phi0}}

 We may assume $\int\,d\mu \neq 0$ since otherwise by the easy fact $\ell^2$ embeds into $\ell^{2,\infty}$, $\lambda\sqrt{N_\lambda(\mathscr U_bf)(x)}$ is pointwisely dominated by the square function $C\|b\|_{L^\infty}(\sum_{k\in \Bbb Z}|\mu_k\ast f(x)|^2)^{1/2}$, and known bounds from \cite{DR86} apply.
Therefore we may normalized $\mu$ so that $\int\,d\mu=1.$ Let $\omega$ be a smooth function with compact support such that $\int_{{\Bbb R}^n}\omega(x)\,dx=1$ and decomposes $\mu=\omega\ast\mu+(\delta_0-\omega)\ast \mu$  where $\delta_0$ is the Dirac mass at $0.$ This in turn decompose $\Phi_kb\Upsilon_kf$ into low and high frequency families $\mathcal{L}=\{\mathcal{L}_k\}$ and $\mathcal{H}=\{\mathcal{H}_k\},$ where $\mathcal{L}_kf(x)=\Phi_kb(x)(\omega\ast\mu)_k\ast f(x)$ and $\mathcal{H}_kf(x)=\Phi_kb(x)[\mu\ast(\delta_0-\omega)]_k\ast f(x).$ By the quasi-triangle inequality, it suffices to bound $\lambda\sqrt{N_\lambda(\mathcal{L}f)}$ and $\lambda\sqrt{N_\lambda(\mathcal{H}f)}$ separately.
Since $\mu\ast(\delta_0-\omega)$ has vanishing mean value and satisfies condition \eqref{condition}, we recall from \cite{DR86} that the square function $$g(f)(x)=\big(\dsum_{k\in \Bbb Z}\mathcal{H}_kf(x)|^2\big)^{1/2}$$ satisfies $$\|g(f)\|_{L^p}\le C\|f\|_{L^p}$$ for $1<p<\infty.$ Furthermore, if $\mu$ satisfies the stronger hypothesis \eqref{condition1}, we can also get weak type $(1,1) $ bounds  for $g(f)$. The easy fact $\ell^2$ embeds into $\ell^{2,\infty}$ implies  $$\lambda\sqrt{N_\lambda(\mathcal{H} f)(x)}\le C\|b\|_{L^\infty}g(f)(x),$$ so matters are reduced to bounding  $\lambda\sqrt{N_\lambda(\mathcal{L}f)}$. We need to prove that  for $1<p<\infty$,
\begin{equation}\label{p}
\|\lambda\sqrt{N_\lambda(\mathcal{L}f)}\|_{L^p}\le C\| {b}\|_{L^\infty}\|f\|_{L^p}
\end{equation}  and \begin{equation}\label{weak}\alpha|\{x\in {\Bbb R}^n: \lambda\sqrt{N_\lambda(\mathcal{L}f)(x)}>\alpha\}|\le C\|b\|_{L^\infty}\|f\|_{L^1}\end{equation} uniformly in $\lambda>0.$
Denote by $\Gamma_kf(x)=(\omega\ast\mu)_k\ast f(x).$ In the following, we will divided the proof into two cases: Case 1, $\Phi_k1\neq 0;$  Case 2, $\Phi_k1= 0.$

\textbf{
Case 1, $\Phi_k1\neq 0.$} By normalization, $\Phi_k1$ can be assumed to $1$.
Then write
\begin{align}\label{decompose}\Phi_kb\Gamma_kf&=\Phi_kb(\Gamma_kf-\mathbb E_kf)+(\Phi_kb-\mathbb E_kb)\mathbb E_kf+\mathbb E_kb\mathbb E_kf\\&\nonumber:=W_k^1f+W_k^2f+W_k^3f.\end{align}
By subadditivity,
\begin{equation}\nonumber
\lambda\sqrt{N_\lambda(\mathcal{L}f)}\le C\lambda\sqrt{N_{\lambda/3}(\{W_k^1f\}_k)}+C\lambda\sqrt{N_{\lambda/3}(\{W_k^2f\}_k)}+C\lambda\sqrt{N_{\lambda/3}(\{W_k^3f\}_k)}.
\end{equation}
To bound  $\lambda\sqrt{N_\lambda(\mathcal{L}f)},$  we first need to prove $L^2$ norm of the above three parts and then weak $(1,1)$-norm of $\lambda\sqrt{N_\lambda(\mathcal{L}f)}$. For $W_k^1f$,
by Lemma 3.2 in \cite{JSW08}, we have for $1<p<\infty,$ \begin{align}\label{w1p}\bigg\|\lambda\sqrt{N_\lambda(\{W_k^1f\}_k)}\bigg\|_{L^p}&\le C\bigg\|\bigg(\dsum_{k\in \Bbb Z}|\Phi_kb(\Gamma_kf-\mathbb E_kf)|^2\bigg)^{1/2}\bigg\|_{L^p}\\&\le \nonumber C\|\Phi_kb\|_{L^\infty}\bigg\|\bigg(\dsum_{k\in \Bbb Z}|\Gamma_kf-\mathbb E_kf|^2\bigg)^{1/2}\bigg\|_{L^p}\\&\le \nonumber C\|b\|_{L^\infty}\|f\|_{L^p}.\end{align}

For $W_k^2f.$
 Let $F(x, 2^k)=\mathbb E_kf(x)$. Define $F^*(x)=\sup_{k>0}\sup_{y\in {\Bbb R}^n \atop |y-x|<2^k}|F(y, 2^k)|$.
  It is easy to see that $|F^*(x)|\le CMf(x),$  where $M$  is the Hardy-Littlewood maximal operator. Then by Lemma \ref{Phib} and Carleson's inequality (see \cite{FS}), we get
\begin{align}\label{w12}\bigg\|\lambda\sqrt{N_\lambda(\{W_k^2f\}_k)}\bigg\|_{L^2}^2&\le C\dsum_{k\in \Bbb Z}\dint_{{\Bbb R}^n} |\Phi_kb(x)-\mathbb E_kb(x)|^2|\mathbb E_kf(x)|^2\,dx\\&\le \nonumber C\|b\|_{\ast}^2\|M  f\|_{L^2}^2\le C\|b\|_{L^\infty}^2\|  f\|_{L^2}^2.\end{align}

To deal with the third term $W_k^3f$, we need the following lemma.

\begin{lemma}\label{lem:finding stop time}
Fix $\lambda>0$.
For a.e. $x\in\mathbb R^n$, we can find a sequence of decreasing stopping times $\{t_i\}_{i\geq0}$ such that
\begin{align}\label{finding stop time}
\lambda\sqrt{N_{\lambda}(\{W^3_kf\}_{k\in\mathbb Z})(x)}\leq 2\bigg(\dsum_{i\geq0}|W^3_{t_{i+1}}f(x)-W^3_{t_{i}}f(x)|^2\bigg)^{1/2}.
\end{align}
\end{lemma}

\begin{proof}
Since $f\in L^p(\mathbb R^n)$, $b\in L^\infty(\mathbb R^n)$, by Jensen inequality we have
\begin{align*}
\sup_{x\in\mathbb R^n}|W^3_kf(x)|&\leq \sup_{x\in\mathbb R^n}|\mathbb E_kf(x)\mathbb E_kb(x)|\\
&\leq \sup_{x\in\mathbb R^n}(\mathbb E_k|f|^p)^{\frac1p}(x)|\mathbb E_kb(x)|\leq 2^{\frac{-kn}{p}}\|f\|_{L^p}\|b\|_{L^\infty}.
\end{align*}
Let $K$ be the smallest integer such that $2^{\frac{-Kn}{p}}\|f\|_{L^p}\|b\|_{L^\infty}\leq \lambda/4.$ Since $W^3_kf$ is $k$-th measurable, that is, constant-valued on the atoms of $\mathcal D_k$, we can construct a sequence of decreasing stopping times $\{t_i\}_{i\geq0}$ as follows. Let $t_0=K$. For $i\geq1$, $t_i$ is constructed inductively
$$t_i=\sup\{j:\;|W^3_jf-W^3_{t_{i-1}}f|>\frac{\lambda}{2}\}.$$

From previous estimates, for all $x\in\mathbb R^n$, $W^3_kf(x)$ converges to zero as $k\rightarrow\infty$; By standard arugments---maximal inequality and Banach principle, it is also easy to see $W^3_kf$ converges a.e. as $k\rightarrow-\infty$. Hence for a.e. $x\in\mathbb R^n$, $N_{\lambda}(\{W^3_kf\}_{k\in\mathbb Z})(x)$ is finite. Fix $x\in\mathbb R^n$, assume $N_{\lambda}(\{W^3_kf\}_{k\in\mathbb Z})(x)=N$, which means there exists a sequence of integers $\{k_i\}_{0\leq i\leq N}$ such that $|W^3_{k_{i+1}}f(x)-W^3_{k_i}f(x)|>{\lambda}$.
Then $|W^3_{k_{1}}f(x)-W^3_{k_0}f(x)|>{\lambda}$ implies either $|W^3_{k_{1}}f(x)-W^3_{t_0}f(x)|>\frac{\lambda}{2}$ or $|W^3_{k_{0}}f(x)-W^3_{t_0}f(x)|>\frac{\lambda}{2}$. By the defintion of $t_1$, we have $t_1(x)\geq k_1$. Inductively, we have $t_i(x)\geq k_i$ for all $1\leq i\leq N$. Thus
\begin{align*}
\dsum_{i\geq0}|W^3_{t_{i+1}}f(x)-W^3_{t_{i}}f(x)|^2&\geq \dsum_{0\leq i\leq N-1}|W^3_{t_{i+1}}f(x)-W^3_{t_{i}}f(x)|^2\\
&\geq N(\lambda/2)^2=(\lambda/2)^2N_{\lambda}(\{W^3_kf\}_{k\in\mathbb Z})(x),
\end{align*}
which yields the desired result.
\end{proof}

 Now we deal with $W_k^3f$. By Lemma \ref{lem:finding stop time},  we can find a sequence of stopping times $\{t_k\}_{k\geq0}$ such that
 $$\begin{array}{cl}
\|\lambda\sqrt{N_\lambda(\{W_k^3f\}_{k\in \Bbb Z})}\|_{L^2}&\le 2\bigg\|\bigg(\dsum_{k\ge 0}|\mathbb E_{t_{k+1}}b\mathbb E_{t_{k+1}}f-\mathbb E_{t_{k}}b\mathbb E_{t_{k}}f|^2\bigg)^{1/2}\bigg\|_{L^2}\\&\le 2\bigg\|\bigg(\dsum_{k\ge 0}|(\mathbb E_{t_{k+1}}b-\mathbb E_{t_{k}} b)\mathbb E_{t_{k}}f|^2\bigg)^{1/2}\bigg\|_{L^2}\\&+2\bigg\|\bigg(\dsum_{k\ge 0}|(\mathbb E_{t_{k+1}}f-\mathbb E_{t_{k}}f)\mathbb E_{t_{k+1}}b|^2\bigg)^{1/2}\bigg\|_{L^2}.
\end{array}$$
By Lemma \ref{Phib1} and Carleson's inequality (see \cite{FS} ), we get
 \begin{align}\label{w32}\dsum_{k\ge 0}\dint_{{\Bbb R}^n} |\mathbb E_{t_{k+1}}b(x)-\mathbb E_{t_{k}} b(x)|^2|\mathbb E_{t_{k}}f(x)|^2\,dx&\le C\|b\|_{\ast}^2\|M  f\|_{L^2}^2\\&\le \nonumber C\|b\|_{L^\infty}^2\|  f\|_{L^2}^2.\end{align}
Since $\|\mathbb E_{t_{k}}b\|_{L^\infty}\le \|b\|_{L^\infty}$ and $\{E_{t_{k}}f\}_{k\geq0}$ is still a martingale (see for instance \cite{PiXu88}), using Burkholder-Gundy inequality, we get for $1<p<\infty$
 \begin{align}\label{w3p}&\bigg\|\bigg(\dsum_{k\ge 0} |\mathbb E_{t_{k+1}}f-\mathbb E_{t_{k}} f|^2|\mathbb E_{t_{k+1}}b|^2\bigg)^{1/2}\bigg\|_{L^p}\\&\le  \nonumber C\|b \|_{L^\infty}\bigg\|\bigg(\dsum_{k\ge 0} |\mathbb E_{t_{k+1}}f-\mathbb E_{t_{k}} f|^2\bigg)^{1/2}\bigg\|_{L^p}\\&\le  \nonumber C\|b\|_{L^\infty}\|  f\|_{L^p}.\end{align}
Combining the estimates of \eqref{w32} and  \eqref{w3p} for $p=2,$ we get  $$\begin{array}{cl}
\|\lambda\sqrt{N_\lambda(\{W_k^3f\}_{k\in \Bbb Z})}\|_{L^2}&\le C\|b\|_{L^\infty}\|  f\|_{L^2}.
\end{array}$$
Combing the estimates of $W_k^if,\,i=1,2,3$, we get  \begin{equation}\label{2}
\|\lambda\sqrt{N_\lambda(\mathcal{L}f)}\|_{L^2}\le C\|b\|_{L^\infty}\|f\|_{L^2}.
\end{equation}
 Next we apply \eqref{2} to establish weak type $(1,1)$ bounds for $\lambda\sqrt{N_\lambda(\mathcal{L}f)}.$ To establish \eqref{weak}
we  perform
the Calder\'{o}n-Zygmund  decomposition of $f$ at height  $\alpha$,  producing a disjoint  family
of dyadic cubes ${Q}$ with  total  measure
$\dsum  |Q| \le  \frac{C}{\alpha}\|f\|_1$ and allowing us  to  write $f=g+h$ with $\|g\|_{L^\infty}\le C\alpha$, $\|g\|_{L^1}\le C\|f\|_{L^1}$ and $ h=\sum_{Q}h_Q,$
where  each $h_Q$  is supported in $Q$
and  has mean  value  zero  such  that
$\sum\|h_Q\|_1\le C\|f\|_{1}.$ Since  we already know  that the  $L^2$  norm  of
$\lambda\sqrt{N_\lambda(\mathcal{L}g)}$ is uniformly controlled by the $L^2$  norm  of $g$,
matters are  reduced  in the usual way to estimating  $\lambda\sqrt{N_\lambda(\mathcal{L}h)}$
away  from $\bigcup \widetilde{Q}$ where $\widetilde{Q}$ is  a  fixed large dilate  of
$Q$. The fact $\ell^1$ embeds into $\ell^2$ implies
$$\begin{array}{cl}\lambda\sqrt{N_\lambda(\mathcal{L}h)(x)}&\le 2\dsum_{k\in \Bbb Z}|\Phi_{k}b(x)\Gamma_{k}h(x)|,\end{array}$$
we see that
$$\begin{array}{cl}&\alpha|\{x\notin\bigcup \widetilde{Q}:\lambda\sqrt{N_\lambda(\mathcal{L}h)}>\alpha\}|\\&\le 2\dsum_{Q}\dsum_{k\in \Bbb Z}\dint_{x\notin\widetilde{Q}}|\Phi_{k}b(x)\Gamma_{k}h_Q(x)|\,dx\\&\le 2\dsum_{Q}\dsum_{k<k(Q)}\dint_{x\notin\widetilde{Q}}|\Phi_{k}b(x)\Gamma_{k}h_Q(x)|\,dx+ 2\dsum_{Q}\dsum_{k\ge k(Q)}\dint_{x\notin\widetilde{Q}}|\Phi_{k}b(x)\Gamma_{k}h_Q(x)|\,dx.\end{array}$$
 For $k\le k(Q)$ (here $2^{k(Q)}$ is roughly the diameter of $Q$ described in Lemma 3.1 in \cite{JSW08}) we estimate
 $$\begin{array}{cl}& \dsum_{Q}\dsum_{k<k(Q)}\dint_{x\notin\widetilde{Q}}|\Phi_{k}b(x)\Gamma_{k}h_Q(x)|\,dx\\&\le C\dsum_{Q}\dsum_{k<k(Q)}\|\Phi_{k}b\|_{L^\infty}\dint_Q|h_Q(y)|\dint_{x\notin \widetilde{Q}}2^{-kn}(2^{-k}|x-y|)^{-(n+1)}\,dxdy\\&\le C\|b\|_{L^\infty}\dsum_{Q}\dsum_{k<k(Q)}\dint_{Q}|h_Q(y)|\dint_{|x-y|\ge C2^ {k(Q)}}2^{-kn}(2^{-k}|x-y|)^{-(n+1)}\,dxdy\\&\le C\|b\|_{L^\infty}\dsum_{Q}\dsum_{k<k(Q)}2^{(k-k(Q))}\|h_Q\|_{L^1}\le C\|b\|_{L^\infty}\|f\|_{L^1}.\end{array}$$
Thus, using the vanishing mean value of $h_Q$, the right side of the above inequality is dominated by $$\begin{array}{cl}&\dsum_{Q}\dsum_{k\ge k(Q)}\dint_{x\notin\widetilde{Q}}|\Phi_{k}b(x)\Gamma_{k}h_Q(x)|\,dx\\&\le\dsum_{Q}\dsum_{k\ge k(Q)}\dint_Q|h_Q(y)|\dint_{x\notin \widetilde{Q}}|(\mu\ast \omega)_k(x-y)-(\mu\ast \omega)_k(x-y_Q)|\,dxdy,\end{array}$$ where $y_Q$ denotes the `center' of $Q$ as described in Lemma 3.1 in \cite{JSW08}). This in turn, using condition \eqref{condition1}, is $$\begin{array}{cl}\dsum_{Q}\dsum_{k\ge k(Q)}\dint_{x\notin\widetilde{Q}}|\Phi_{k}b(x)\Gamma_{k}h_Q(x)|\,dx&\le C\| b\|_{L^\infty}\dsum_{Q}\dsum_{k\ge k(Q)}2^{-\tau(k-k(Q))}\|h_Q\|_{L^1}\\&\le C\|b\|_{L^\infty}\|f\|_{L^1}\end{array}$$ establishing the uniform weak-type (1,1) bound for $\lambda\sqrt{N_\lambda(\mathcal{L}f)}$ and therefore finishing the proof of \eqref{weak}.
 By interpolation between (\ref{2}) and  (\ref{weak}), imply all the $L^p$ bounds $\lambda\sqrt{N_\lambda(\mathcal{L}f)}$ of for $1<p\le2$.
So to prove \eqref{p}, it suffices to prove $L^p$ bounds of $\lambda\sqrt{N_\lambda(\mathcal{L}f)}$  for $2<p<\infty.$
Since we have obtained the $L^p$ bounds of $\lambda\sqrt{N_\lambda(\{W_k^1f\}_k)}$  for $1<p<\infty$ in \eqref{w1p} and the  $L^p$ bounds of $I_2$ for  $1<p<\infty$ in \eqref{w3p}, we need only to prove  for $2<p<\infty$
\begin{align}\label{p2}
\big\|\big(\sum_{k\in\mathbb Z}|(\Phi_kb-\mathbb E_kb)\mathbb E_kf|^2\big)^{1/2}\big\|_{L^p}\le C\|b\|_{L^\infty}\|f\|_{L^p}.
\end{align}
and
\begin{align}\label{p3}
\big\|\big(\sum_{k\ge 0}|(\mathbb E_{t_{k+1}}b-\mathbb E_{t_k}b)\mathbb E_{t_{k+1}}f|^2\big)^{1/2}\big\|_{L^p}\le C\|b\|_{L^\infty}\|f\|_{L^p}.
\end{align}

We first prove \eqref{p2}.
For $2<p<\infty$, by H\"{o}lder's inequality, we have
\begin{align}\label{pe}
&\big\|\big(\sum_{k\in\mathbb Z}|[\Phi_kb-\mathbb E_kb]\mathbb E_kf|^2\big)^{1/2}\big\|_{L^p}\\& \nonumber=\sup_{\|\{h_k\}\|_{L^{p'}(\ell^2)}\le1}\big|\int_{\mathbb R^n}\sum_{k\in \Bbb Z}([\Phi_kb(x)-\mathbb E_kb(x)]\mathbb E_kf(x))h_k(x)dx\big|\\
&\nonumber=\sup_{\|\{h_k\}\|_{L^{p'}(\ell^2)}\le1}\big|\int_{\mathbb R^n}\sum_{k\in \Bbb Z}[\Phi_k(\mathbb E_kf\cdot h_k)(y)-\mathbb E_k(\mathbb E_kf \cdot h_k)(y)]b(y)dy\big|\\
&\le\nonumber\sup_{\|\{h_k\}\|_{L^{p'}(\ell^2)}\le1}\|\sum_{k\in \Bbb Z}[\Phi_k(\mathbb E_kf \cdot h_k)-\mathbb E_k(\mathbb E_kf\cdot h_k)]\|_{L^{1}}\|b\|_{L^\infty}.
\end{align}
It suffices to show that
\begin{align}\label{sdp'}
\|\sum_{k\in \Bbb Z}[\Phi_k(\mathbb E_kf \cdot h_k)-\mathbb E_k(\mathbb E_kf\cdot  h_k)]\|_{L^{1}}\le C\|f\|_{L^p}\|\{h_k\}\|_{L^{p'}(\ell^2)},\ \ 1<p'\le 2.
\end{align}
Clearly,  using $\|\big(\sum_{k\in\mathbb Z}|(\Phi_kb-\mathbb E_kb)\mathbb E_kf|^2\big)^{1/2}\|_{L^2}\le C\|b\|_{L^\infty}\|f\|_{L^2}$ (see \eqref{w12}) by duality,
\begin{align}\label{sd2}
\|\sum_{k\in \Bbb Z}[\Phi_k(\mathbb E_kf\cdot h_k)-\mathbb E_k(\mathbb E_kf\cdot h_k)]\|_{L^1}\le C\|f\|_{L^2}\|\{h_k\}\|_{L^2(\ell^2)}.
\end{align}
Applying
$\big|\{x\in\mathbb R^n:|\sum_{k\in \Bbb Z}[\Phi_k(g_k)(x)-\mathbb E_k(g_k)(x)]|>\alpha\}\big|\le \frac C{\alpha}\|\{g_k\}\|_{L^1(\ell^2)},
$ which was established in \cite{DHL} and $|\mathbb E_kf(x)|\le \|f\|_{L^\infty}$ for any fixed $x\in {\Bbb R}^n$, we get
\begin{align}\label{sd1}\nonumber
\big|\{x\in\mathbb R^n:|\sum_{k\in \Bbb Z}[\Phi_k(\mathbb E_kf \cdot h_k)(x)-\mathbb E_k(\mathbb E_kf \cdot h_k)(x)]|>\alpha\}\big|&\le \frac C{\alpha}\|\{\mathbb E_kf \cdot h_k\}\|_{L^1(\ell^2)}\\&\le \frac C{\alpha}\|f\|_{L^\infty}\|\{h_k\}\|_{L^1(\ell^2)},
\end{align}
where $\alpha>0$ and $C$ is independent of $\alpha$, $f$ and $\{h_k\}$.
Then by interpolation between \eqref{sd2} and \eqref{sd1}, we get \eqref{sdp'}.

Next we prove \eqref{p3}. Similar to the proof of \eqref{pe}, we get for $2<p<\infty,$
 \begin{align*}
&\big\|\big(\sum_{k\ge 0}|[\mathbb E_{t_{k+1}}b-\mathbb E_{t_k}b]\mathbb E_{t_{k+1}}f|^2\big)^{1/2}\big\|_{L^p}\\
&\le\sup_{\|\{h_k\}\|_{L^{p'}(\ell^2)}\le1}\Big\|\sum_{k\ge 0}[\mathbb E_{t_{k+1}}(\mathbb E_{t_{k+1}}f\cdot h_k)-\mathbb E_{t_k}(\mathbb E_{t_{k+1}}f \cdot h_k)]\Big\|_{L^{1}}\|b\|_{L^\infty}.
\end{align*}
It suffices to show that
\begin{align}\label{Edp'}
\Big\|\sum_{k\ge 0}[\mathbb E_{t_{k+1}}(\mathbb E_{t_{k+1}}f \cdot h_k)-\mathbb E_{t_k}(\mathbb E_{t_{k+1}}f \cdot h_k)]\Big\|_{L^{1}}\le C\|f\|_{L^p}\|\{h_k\}\|_{L^{p'}(\ell^2)},\ \ 1<p'\le 2.
\end{align}
By $\|\big(\sum_{k\ge 0}|[\mathbb E_{t_{k+1}}b-\mathbb E_{t_k}b]\mathbb E_kf|^2\big)^{1/2}\|_{L^2}\le C\|b\|_{L^\infty}\|f\|_{L^2}$ (see \eqref{w32}) by duality, we get
\begin{align}\label{Ed2}
\|\sum_{k\ge 0}[\mathbb E_{t_{k+1}}(\mathbb E_{t_{k+1}}f\cdot h_k)-\mathbb E_{t_k}(\mathbb E_{t_{k+1}}f\cdot h_k)]\|_{L^1}\le C\|f\|_{L^2}\|\{h_k\}\|_{L^2(\ell^2)}.
\end{align}
if we can prove that for $\{\widetilde{h}_k\}\in L^1(\ell^2)({\Bbb R}^n)$, \begin{align}\label{Ed1}
\big|\{x\in\mathbb R^n:|\sum_{k\ge 0}[\mathbb E_{t_{k+1}}( \widetilde{h}_k)-\mathbb E_{t_k}(\widetilde{ h}_k)(x)]|>\alpha\}\big|\le \frac C{\alpha}\|\{\widetilde{h}_k\}\|_{L^1(\ell^2)},
\end{align}
then by $|\mathbb E_{t_{k+1}}f(x)|\le \|f\|_{L^\infty}$ for any fixed $x\in {\Bbb R}^n$, we can  get
\begin{align}\label{Ed3}\nonumber
\big|\{x\in\mathbb R^n:|\sum_{k\ge 0}[\mathbb E_{t_{k+1}}(E_{t_{k+1}}f \cdot h_k)(x)-\mathbb E_{t_k}(E_{t_{k+1}}f \cdot h_k)(x)]|>\alpha\}\big|&\le \frac C{\alpha}\|\{E_{t_{k+1}}f\cdot h_k\}\|_{L^1(\ell^2)}\\&\le \frac C{\alpha}\|f\|_{L^\infty}\|\{h_k\}\|_{L^1(\ell^2)},
\end{align}
where $\alpha>0$ and $C$ is independent of $\alpha$, $f$ and $\{h_k\}$.
Thus, by interpolation between (\ref{Ed2}) and (\ref{Ed3}), we get (\ref{Edp'}).

 Now we prove \eqref{Ed1}. For $\alpha>0$, we perform Calder\'{o}n-Zygmund decomposition of $\|\{\widetilde{h}_k\}\|_{\ell^2}$ at height $\alpha$, then there exists $\Lambda\subseteq\mathbb Z\times\mathbb Z^n$ such that the collection of dyadic cubes $\{Q_\beta^j\}_{(j,\beta)\in\Lambda}$ are disjoint and the following hold:
  \begin{itemize}
  \item[{\rm(i)}]
  $|\bigcup_{(j,\beta)\in\Lambda}Q_\beta^j|\le \alpha^{-1}\|\{\widetilde{h}_k\}\|_{L^1(\ell^2)}$;
  \item[{\rm(ii)}]
  $\|\{\widetilde{h}_k(x)\}\|_{\ell^2}\le \alpha$, if $x\not\in\bigcup_{(j,\beta)\in\Lambda}Q_\beta^j$;
   \item[{\rm(iii)}]
   $\frac1{|Q_\beta^j|}\int_{Q_\beta^j}\|\{\widetilde{h}_k(x)\}\|_{\ell^2}dx\le2^n\alpha$ for each $(j,\beta)\in\Lambda$.
   \end{itemize}
 For $k\in \Bbb Z$, we set
\begin{equation}\nonumber
 g^{(k)}(x)=\left\{
 \begin{array}{ll}
 \widetilde{h}_k(x),& \text{if}\ x\not\in \bigcup_{(j,\beta)\in\Lambda}Q_\beta^j,\\
 \frac1{|Q_\beta^j|}\int_{Q_\beta^j}\widetilde{h}_k(y)dy,& \text{if}\ x\in Q_\beta^j,(j,\beta)\in\Lambda.
 \end{array}
 \right.
 \end{equation}
 and
 \begin{equation}\nonumber
 e^{(k)}(x)=\sum_{(j,\beta)\in\Lambda}[\widetilde{h}_k(x)-\mathbb E_j\widetilde{h}_k(x)]\chi_{Q_\beta^j}(x):=\sum_{(j,\beta)\in\Lambda} e^{(k)}_{j,\beta}(x).
 \end{equation}
 First we have $\|\{g^{(k)}\}\|^2_{L^2(\ell^2)}\le 2\alpha\|\{\widetilde{h}_k\}\|_{L^1(\ell^2)}$. In fact, by (ii), (iii) and Minkowski's inequality,
 \begin{align*}
 \|\{g^{(k)}\}\|^2_{L^2(\ell^2)}&=\int_{(\cup_{(j,\beta)\in\Lambda}Q_\beta^j)^c}\|\{\widetilde{h}_k(x)\}\|^2_{\ell^2}dx+
 \sum_{(j,\beta)\in\Lambda}\int_{Q_\beta^j}\sum_{k\in \Bbb Z}\big|\frac1{|Q_\beta^j|}\int_{Q_\beta^j}\widetilde{h}_k(y)dy\big|^2dx\\
 &\le \alpha\int_{(\cup_{(j,\beta)\in\Lambda}Q_\beta^j)^c}\|\{\widetilde{h}_k(x)\}\|_{\ell^2}dx+2^n\alpha\sum_{(j,\beta)\in\Lambda}\int_{Q_\beta^j}\|\widetilde{h}_k(x)\|_{\ell^2}dx\\
 &\le 2^n\alpha\|\{\widetilde{h}_k\}\|_{L^1(\ell^2)}.
 \end{align*}
 Thus, for above $\alpha$, by the result in  \cite{PiXu88} by duality,

\begin{align*}
&\alpha^2\big|\{x\in \mathbb R^n:|\sum_{k\ge 0}[\mathbb E_{t_{k+1}} g^{(k)}(x)-\mathbb E_{t_{k}}g^{(k)}(x)]|>\alpha\}\big|\\&\le C\big\|\sum_{k\ge 0}[\mathbb E_{t_{k+1}}g^{(k)}-\mathbb E_{t_{k}}g^{(k)}\big\|_{L^2}^2\\
&\le C\|\{g^{(k)}\}\|_{L^2(\ell^2)}^2\\
&\le C\|\{g^{(k)}\}\|_{L^2(\ell^2)}^2\le C\alpha\|\{\widetilde{h}_k\}\|_{L^1(\ell^2)}.
\end{align*}
So, we get
\begin{equation}\nonumber
\big|\{x\in \mathbb R^n:|\sum_{k\ge 0}[\mathbb E_{t_{k+1}} g^{(k)}(x)-\mathbb E_{t_{k}}g^{(k)}(x)]|\le \frac C{\alpha}\|\{\widetilde{h}_k\}\|_{L^1(\ell^2)}.
\end{equation}
  On the other hand, it is easy to see that
 $$\int_{\Bbb R^n} e^{(k)}_{j,\beta}(x)dx=0\quad \text{for all}\quad k\in\mathbb Z,\ (j,\beta)\in\Lambda.$$
Let $\tilde{Q}_\beta^j$ be the cube concentric with $Q_\beta^j$ and with side length $4$ times that of $Q_\beta^j$. It is obvious that
\begin{equation}
\big|\bigcup_{(j,\beta)\in\Lambda}\tilde{Q}_\beta^j\big|\le C\sum_{(j,\beta)\in\Lambda}|Q_\beta^j|\le \frac C{\alpha}\|\{\widetilde{h}_k\}\|_{L^1(\ell^2)}.
\end{equation}
Note that $\mathbb E_\ell  e^{(k)}_{j,\beta}$ is supported in $Q_\beta^j$ when $\ell\le j$ and $\mathbb E_\ell  e^{(k)}_{j,\beta}$ vanishes everywhere when $\ell\ge j$.
\begin{align*}
&\alpha\big|\{x\not\in\bigcup\tilde{Q}_\beta^j:|\sum_{k\ge 0}[\mathbb E_{t_{k+1}}( e^{(k)})(x)-\mathbb E_{t_{k}}( e^{(k)})(x)]|>\alpha\}\big|=0.
\end{align*}
This completes the proof of \eqref{Ed1}.

\textbf{Case 2, $\Phi_k1= 0.$} The argument is very similar to the proof of Case $1$ but easier. Since $\phi\in {\mathscr
S}({\Bbb R}^n)$ and $\widehat{\phi}(0)=0$, then $\sup_{k\in \Bbb Z}\|\Phi_kb\|_{L^\infty}\le C\|b\|_{L^\infty}$ and  $d\nu(x,t)=\sum_{k\in \Bbb Z}|\Phi_kb(x)|^2\,dx\,\delta_{2^k}(t)$ is a Carleson measure on $\Bbb R^{n+1}_{+}$ whose norm is controlled by a constant multiple of $\|b\|_{L^\infty}^2$ (see \cite{GR}). So,  we need only a little adjustment in \eqref{decompose} with replacing $\Phi_kb\Gamma_kf=\Phi_kb(\Gamma_kf-\mathbb E_kf)+(\Phi_kb-\mathbb E_kb)\mathbb E_kf+\mathbb E_kb\mathbb E_kf$ by $\Phi_kb\Gamma_kf=\Phi_kb(\Gamma_kf-\mathbb E_kf)+\Phi_kb\mathbb E_kf.$ \qed

\section{Proof of Theorem \ref{Phi1}}
  Write $$[b, \Phi_k]\nabla=[b,\nabla \Phi_k]-\Phi_k[b,\nabla].$$ By subadditivity,
\begin{equation}\nonumber
\lambda\sqrt{N_\lambda(\mathscr F_b \nabla f)}\le C\lambda\sqrt{N_{\lambda/2}(\{[b,\nabla \Phi_k]f\}_k)}+C\lambda\sqrt{N_{\lambda/2}(\{\Phi_k[b,\nabla]f\}_k)}.
\end{equation}
 By Theorem 1.1 in \cite{JSW08},  notice that $[b,\nabla]f=-f\nabla b$, we get
 \begin{equation*}
\|\lambda\sqrt{N_\lambda(\{\Phi_k[b,\nabla]f\}_k)}\|_{L^p}\le C\|[b,\nabla]f\|_{L^p}\le C\|\nabla b\|_{L^\infty}\|f\|_{L^p},\,\,\,\, 1<p<\infty
\end{equation*} and
\begin{equation*}\alpha|\{x\in {\Bbb R}^n: \lambda\sqrt{N_\lambda(\{\Phi_k[b,\nabla]f\}_k)(x)}>\alpha\}|\le C\|[b,\nabla]f\|_{L^1}\le C\|\nabla b\|_{L^\infty}\|f\|_{L^1}\end{equation*} uniformly in $\lambda>0.$
So to prove   that $\lambda\sqrt{N_\lambda(\mathscr F_b \nabla f)}$ is bounded on $L^p({\Bbb R}^n)$ and is of weak $(1,1)$, it suffices to prove  the same properties hold for $\{[b,\nabla \Phi_k]\}_k.$
Write $$[b,\nabla \Phi_k]f=[b,\nabla\Phi_k]f-([b,\nabla \Phi_k]1)\Phi_kf+([b,\nabla \Phi_k]1)\Phi_kf:=P_kf+([b,\nabla \Phi_k]1)\Phi_kf.$$
By subadditivity again,
\begin{equation}\nonumber
\lambda\sqrt{N_\lambda(\{[b,\nabla \Phi_k]f\}_k)}\le C\lambda\sqrt{N_{\lambda/2}(\{P_kf\}_k)}+C\lambda\sqrt{N_{\lambda/2}(\{([b,\nabla \Phi_k]1)\Phi_kf\}_k)}.
\end{equation}
For $\{P_kf\}_{k\in \Bbb Z}$, we need the following lemma.

\begin{lemma}\label{ChJ} \,(\cite{ChJ, DR1, GR}) Denote by $\Theta_j f(x):=\int_{\mathbb{R}^n}\psi_j(x,y)f(y)\,dy,$ where $\psi_j(x,y)$ satisfies the standard kernel conditions, i.e., for some $\gamma>0$ and $C>0,$ \begin{align}\label{ChJ1}|\psi_j(x,y)|\le C\dfrac{2^{j\gamma}}{(2^j+|x-y|)^{n+\gamma}}\end{align}
and
\begin{align}\label{ChJ2}|\psi_j(x+h,y)-\psi_j(x,y)|+|\psi_j(x,y+h)-\psi_j(x,y)|\le C\dfrac{|h|^\gamma}{(2^j+|x-y|)^{n+\gamma}},\quad |h|\le 2^j,\end{align}  for all $x,\,y\in \mathbb{R}^n$ and $j\in \Bbb Z.$ If $\Theta_j 1=0$, then for $1<p<\infty,$
$$\big\|(\sum_{j\in \Bbb Z}|\Theta_j f|^2)^{1/2}\big\|_{L^p}\le C\|f\|_{L^p}$$ and \begin{equation*}\sup_{\alpha>0}\alpha|\{x\in {\Bbb R}^n: (\sum_{j\in \Bbb Z}|\Theta_j f(x)|^2)^{1/2}>\alpha\}|\le C\|f\|_{L^1}.\end{equation*}\end{lemma}
Denote by $\widetilde{\phi}:=\nabla\phi$ and $\widetilde{b}:=\nabla b$. Then we can write $\nabla \Phi_k f=2^{-k}\widetilde{\phi}_k\ast f$
and $[b,\nabla \Phi_k]1=-\phi_k\ast\widetilde{b}.$ Recall that $P_kf=[b,\nabla\Phi_k]f-([b,\nabla \Phi_k]1)\Phi_kf.$ Let $\psi_k(x,y)$ be the kernel of the operator $P_k$ with  \begin{align*}P_kf(x)=\int_{{\Bbb R}^n}\psi_k(x,y)f(y)\,dy.\end{align*}Then we can write \begin{align*}\psi_k(x,y)=2^{-k}\widetilde{\phi}_k(x-y)(b(x)-b(y))+(\phi_k\ast\widetilde{b})(x)\phi_k(x-y).\end{align*} By $|b(x)-b(y)|\le \|\widetilde{b}\|_{L^\infty}|x-y|$ and $|(\phi_k\ast \widetilde{b})(x)|\le \|\widetilde{b}\|_{L^\infty}$, we get
\begin{align*}|\psi_k(x,y)|\le 2^{-k}\|\widetilde{b}\|_{L^\infty}|\widetilde{\phi}_k(x-y)||x-y|+\|\widetilde{b}\|_{L^\infty}|\phi_k(x-y)|\le C\|\widetilde{b}\|_{L^\infty}\frac{2^k}{(2^k+|x-y|)^{n+1}}\end{align*}for all $x,\,y\in \mathbb{R}^n$ and $k\in \Bbb Z.$ Also by $\widetilde{\phi}\in \mathscr{S}({\Bbb R}^n)$, $|b(x)-b(y)|\le \|\widetilde{b}\|_{L^\infty}|x-y|$ and $|(\phi_k\ast \widetilde{b})(x)|\le \|\widetilde{b}\|_{L^\infty}$, we get \begin{align*}|\psi_k(x,y+h)-\psi_k(x,y)|&\le 2^{-k}|\widetilde{\phi}_k(x-y-h)-\widetilde{\phi}_k(x-y)||b(x)-b(y)|\\&+2^{-k}|\widetilde{\phi}_k(x-y-h)||b(y)-b(y+h)|\\&+|(\phi_k\ast \widetilde{b})(x)||\phi_k(x-y-h)-\phi_k(x-y)|\\&\le C\|\widetilde{b}\|_{L^\infty}\frac{|h|}{(2^k+|x-y|)^{n+1}},\,\,\,\,\,\,\,\, |h|\le 2^k,\end{align*} for all $x,\,y\in \mathbb{R}^n$ and $k\in \Bbb Z.$
Similarly, we get\begin{align*}|\psi_k(x+h,y)-\psi_k(x,y)|&\le 2^{-k}|\widetilde{\phi}_k(x+h-y)-\widetilde{\phi}_k(x-y)||b(x)-b(y)|\\&+2^{-k}|\widetilde{\phi}_k(x+h-y)||b(x+h)-b(x)|\\&+|(\phi_k\ast \widetilde{b})(x)||\phi_k(x+h-y)-\phi_k(x-y)|\\&+|(\phi_k\ast \widetilde{b})(x+h)-(\phi_k\ast \widetilde{b})(x)||\phi_k(x+h-y)|\\&\le C\|\widetilde{b}\|_{L^\infty}\frac{|h|}{(2^k+|x-y|)^{n+1}},\,\,\,\,\,\,\, |h|\le 2^k,\end{align*} for all $x,\,y\in \mathbb{R}^n$ and $k\in \Bbb Z.$
This says that the kernel of $P_k$ continues to satisfy \eqref{ChJ1} and \eqref{ChJ2}. It is easy to verify that $P_k1=0  $ for all $k\in \Bbb Z$.  Thus by Lemma \ref{ChJ}, we get for $1<p<\infty$
$$\bigg\|\bigg(\sum_{k\in \Bbb Z}|P_k f|^2\bigg)^{1/2}\bigg\|_{L^p}\le C\|\nabla b\|_{L^\infty}\|f\|_{L^p}$$  and the weak type $(1,1)$ estimates for $(\sum_{k\in \Bbb Z}|P_k f|^2)^{1/2}.$ The easy fact $\ell^2$ embeds into $\ell^{2,\infty}$ implies $$\lambda\sqrt{N_\lambda(\{P_kf\}_{k\in \Bbb Z})(x)}\le C\bigg(\sum_{k\in \Bbb Z}|P_k f(x)|^2\bigg)^{1/2},$$
then gives the desired $L^p$ bounds and weak type $(1,1)$ bounds for $\lambda\sqrt{N_\lambda(\{P_kf\}_{k\in \Bbb Z})}.$
On the other hand,  since
$[b,\nabla\Phi_k]1
=-\Phi_k\big(\nabla b\big),
$ then $$([b,\nabla\Phi_k]1)\Phi_kf=-\Phi_k\big(\nabla b\big)\Phi_kf.$$
 Apply Theorem  \ref{Phi0}, we have
\begin{equation}\label{phip}
\|\lambda\sqrt{N_\lambda(\{\Phi_k \big(\nabla b\big)\Phi_kf\}_{k})}\|_{L^p}\le C\| \nabla b\|_{L^\infty}\|f\|_{L^p},\,\,\, 1<p<\infty
\end{equation} and \begin{equation}\label{phiweak}\alpha|\{x\in {\Bbb R}^n: \lambda\sqrt{N_\lambda(\{\Phi_k \big(\nabla b\big)\Phi_kf\}_{k})(x)}>\alpha\}|\le C\|\nabla b\|_{L^\infty}\|f\|_{L^1}\end{equation} uniformly in $\lambda>0.$
Combined these estimates,  we get that $\lambda\sqrt{N_\lambda(\{[b,\nabla \Phi_k]f\}_k)}$ is bounded on $L^p({\Bbb R}^n)$ and is of weak type $(1,1)$ if $b\in Lip({\Bbb R}^n).$

Now we turn to prove that $\lambda\sqrt{N_\lambda(\nabla \mathscr F_bf )}$ is bounded on $L^p({\Bbb R}^n)$ for $1<p<\infty$ and is of weak type $(1,1)$ if $b\in Lip({\Bbb R}^n).$
Write
$$\nabla[b, \Phi_k]f=[b,\nabla \Phi_k]f-[b,\nabla]\Phi_kf=[b,\nabla \Phi_k]f+(\nabla b)\Phi_kf.$$
So, we need only to prove \begin{equation}\label{phip}
\|\lambda\sqrt{N_\lambda(\{ \big(\nabla b\big)\Phi_kf\}_{k})}\|_{L^p}\le C\| \nabla b\|_{L^\infty}\|f\|_{L^p},\,\,\,1<p<\infty
\end{equation} and \begin{equation}\label{phiweak}\alpha|\{x\in {\Bbb R}^n: \lambda\sqrt{N_\lambda(\{ \big(\nabla b\big)\Phi_kf\}_{k})(x)}>\alpha\}|\le C\|\nabla b\|_{L^\infty}\|f\|_{L^1} \end{equation} uniformly in $\lambda>0.$ Note that $\nabla b\in L^\infty({\Bbb R}^n),$  therefore \eqref{phip} and \eqref{phiweak} can be obtained by the very same argument in \cite{JSW08}.
Therefore we finish the proof of Theorem  \ref{Phi1}.\qed

\section{Some more lemmas for Theorem \ref{thm:H}} In this section, we present three more lemmas, which will play a key role in proving Theorem \ref{thm:H}.

\begin{lemma}\label{Deltacom} Let $\varphi\in {\mathscr
S}({\Bbb R}^n)$ be a radial function such that ${\rm supp}\,\varphi
\subset\{1/2\le |\xi|\le 2\}$ and $\widehat{\Delta_jf}(\xi)=\varphi(2^{-j}\xi)\widehat{f}(\xi)$ for $j\in \Bbb Z$.  If $b\in Lip(\mathbb R^n),$ then for  $1<p<\infty$ and
$f\in L^p({\Bbb R}^n)$, we have
$$\bigg\|\bigg(\dsum_{l\in \Bbb
Z}|2^l[b,\Delta_{l}]f|^2\bigg)^{1/2}\bigg\|_{L^p}\le
C_{n,p}\|\nabla b\|_{L^\infty}\|f\|_{L^p}.$$
\end{lemma}

\emph{Proof.} Let $\widehat{\Psi}=\varphi$ and $\Psi_{2^{-j}}(x)=2^{jn}\Psi(2^j x),$ then $\Delta_j f=\Psi_{2^{-j}}\ast f.$  Let
\begin{equation}\nonumber k_j(x,y)=2^{j}(b(x)-b(y))\Psi_{2^{-j}}(x-y).\end{equation}
Define the operator $\mathbb{T}$ by
\begin{equation}\nonumber\mathbb{T}f(x)=\int_{\mathbb{R}^n}\mathbb{K}(x,y)f(y)dy,\end{equation}
where $ \mathbb{K}: (x,y)\rightarrow\{k_j(x,y)\}_{j\in\Bbb Z}$
 with $\|\mathbb{K}(x,y)\|_{\mathbf{\Bbb R}^n\times \mathbf{\Bbb R}^n\rightarrow \ell^2}:=\big(\sum_{j\in \Bbb Z}|k_j(x,y)|^2\big)^{1/2}.$   Lemma 2.3 in \cite{CD1} says that $$\|\mathbb{T}f\|_{L^2(\ell^2)}\leq C \|\nabla b\|_{L^\infty}\|f\|_{L^2}.  $$
  On the other hand, for $b\in Lip({\Bbb R}^n)$, it is easy to verify that for $2|h|\le |x-y|, $\begin{equation}\nonumber\max\bigg\{\bigg(\dsum_{j\in \Bbb Z}|k_j(x,y+h)-k_j(x,y)|^2\bigg)^{1/2},\,\bigg(\dsum_{j\in \Bbb Z}|k_j(x+h,y)-k_j(x,y)|^2\bigg)^{1/2}\bigg\}\le C\| b\|_{Lip}\frac{|h|}{|x-y|^{n+1}}.\end{equation}  Then by the result in \cite{DR1, GR, GR1}, we get the desired result.\qed

\begin{lemma}\label{Tsk} Let $\Omega\in L^1(\Bbb S^{n-1})$ and satisfy the mean value zero property. For $k\in \Bbb Z,$ set
   $\nu_{k}(x)=\frac{\Omega(x)}
   {|x|^{n+1}}\chi_{\{2^k\le |x|
   < 2^{k+1}\}}(x)$ and $T_kf=\nu_{k}\ast f.$ Then we have for $1<p<\infty,$
\begin{align}\nonumber\bigg\|\bigg(\dsum_{k\in\mathbb Z}|T_kf_k|^2\bigg)^{1/2}\bigg\|_{L^p}
&\le
C_{n,p}\|\Omega\|_{L^1(\mathbf S^{n-1})}\biggl\|\bigg(\dsum_{k\in\mathbb Z}|\nabla f_k|^2\bigg)^{1/2}\biggl\|_{L^p}.
\end{align} \end{lemma}
\emph{
Proof.} By the mean value zero property of $\Omega$,  we have for $t\in \Bbb R_{+}$
\begin{align}\nonumber\biggl|\dint_{S^{n-1}}\Omega(y')
f(x-ty')d\sigma(y')\biggl|
&=\biggl|\dint_{S^{n-1}}\Omega(y')
\Big(f(x-ty')-f(x)\Big)d\sigma(y')\biggl|\\
&\le\nonumber\dsum_{|\beta|=1}\dint_0^1\dint_{S^{n-1}}|\Omega(y')|
|D^\beta f(x+sty')|t d\sigma(y')\,ds.
\end{align}
Then, for $\{f_k\}_{k\in\Bbb Z},$ by Lemma 2.3
 in \cite{CD}, we have for $1<p<\infty$
\begin{align} \nonumber\bigg\|\bigg(\dsum_{k\in\mathbb Z}|T_kf_k|^2\bigg)^{1/2}\bigg\|_{L^p}&=\nonumber\biggl\|
\bigg(\dsum_{k\in\mathbb Z}\bigg|\dint_{2^{k-1}}^{2^k}
\biggl|\dint_{S^{n-1}}\Omega(y')
f_k(\cdot-ty')d\sigma(y')\biggl|\frac{dt}{t^{2}}\bigg|^2\bigg)^{1/2}\biggl\|_{L^p}\\
&\leq\nonumber
C\dsum_{|\beta|=1}\dint_0^1\biggl\|\bigg(\dsum_{j\in\mathbb Z}\bigg|\dint_{2^{k-1}}^{2^k}
\biggl|
\dint_{S^{n-1}}|\Omega(y')|
|D^\beta f_k(x+sty')| d\sigma(y')\frac{dt}{t}\bigg|^2\bigg)^{1/2}\biggl\|_{L^p}\,ds\\
&\leq\nonumber
C\dsum_{|\beta|=1}\dint_0^1\biggl\|\bigg(\dsum_{j\in\mathbb Z}\bigg|\dint_{2^{k-1}\le |y|< 2^{k}}
\frac{|\Omega(y')|}{|y|^n}
|D^\beta f_k(x+sy)| dy\bigg|^2\bigg)^{1/2}\biggl\|_{L^p}\,ds\\
&=\nonumber
C\dsum_{|\beta|=1}\dint_0^1\biggl\|\bigg(\dsum_{k\in\mathbb Z}\bigg|\dint_{s2^{k-1}\le |y|< s2^{k}}
\frac{|\Omega(y')|}{|y|^n}
|D^\beta f_k(x+y)| dy\bigg|^2\bigg)^{1/2}\biggl\|_{L^p}\,ds\\
&\le\nonumber
C\dsum_{|\beta|=1}\biggl\|\bigg(\dsum_{k\in\mathbb Z}|M_\Omega (D^\beta f_k)|^2\bigg)^{1/2}\biggl\|_{L^p}\\
&\le\nonumber
C\|\Omega\|_{L^1(\mathbf S^{n-1})}\biggl\|\bigg(\dsum_{k\in\mathbb Z}|\nabla f_k|^2\bigg)^{1/2}\biggl\|_{L^p},
\end{align} where
$$M_{\Omega}f(x)=\sup_{r>0}\frac{1}{r^n}\dint_{|x-y|<r}|\Omega(x-y)||f(y)|\,dy.$$\qed

 \begin{lemma}\label{phicom} Let $\phi\in {\mathscr
S}({\Bbb R}^n)$ and $\Phi_kf(x)=\phi_k\ast f(x),$ where $\phi_k(x)=2^{-kn}\phi(2^{-k}x).$ Then for $1<p<\infty,$ \begin{align}\nonumber\bigg\|\bigg(\dsum_{k\in \mathbb Z}|\nabla[b,\Phi_{k}]f_k|^2\bigg)^{1/2}\bigg\|_{L^p}
&\le
C\|\nabla b\|_{L^\infty}\biggl\|\bigg(\dsum_{k\in\mathbb Z}|f_k|^2\bigg)^{1/2}\biggl\|_{L^p}.
\end{align}\end{lemma}

\emph{
Proof.} Write $\nabla[b,\Phi_{k}]f=[b, \nabla \Phi_k]f-[b,\nabla]\Phi_kf.$ Then by $[b, \nabla]f=-(\nabla b)f$ and $\{\Phi_k\}$ is bounded on $L^p(\ell^2)({\Bbb R}^n),$ we get   \begin{align}\nonumber\bigg\|\bigg(\dsum_{k\in \mathbb Z}|\nabla[b,\Phi_{k}]f_k|^2\bigg)^{1/2}\bigg\|_{L^p}
&\le
C\bigg\|\bigg(\dsum_{k\in \mathbb Z}|[b,\nabla\Phi_{k}]f_k|^2\bigg)^{1/2}\bigg\|_{L^p}+C\bigg\|\bigg(\dsum_{k\in \mathbb Z}|[b,\nabla]\Phi_{k}f_k|^2\bigg)^{1/2}\bigg\|_{L^p}\\&\le
\nonumber C\bigg\|\bigg(\dsum_{k\in \mathbb Z}|[b,\nabla\Phi_{k}]f_k|^2\bigg)^{1/2}\bigg\|_{L^p}+C\|\nabla b\|_{L^\infty}\bigg\|\bigg(\dsum_{k\in \mathbb Z}|f_k|^2\bigg)^{1/2}\bigg\|_{L^p}.
\end{align}
Note that $\nabla\phi_{k}(x)=2^{-k}(\nabla \phi)_k(x)$ and denote by $\nabla\phi=\widetilde{\phi}, $ we get \begin{align}\nonumber|[b,\nabla\Phi_{k}]f_k(x)|&\le 2^{-k(n+1)}\dint_{{\Bbb R}^n}|\widetilde{\phi}(2^{-k}(x-y))||b(x)-b(y)||f_k(y)|\,dy\\&\le \nonumber C\|\nabla b\|_{L^\infty}2^{-k(n+1)}\dint_{{\Bbb R}^n}|\widetilde{\phi}(2^{-k}(x-y))||x-y||f_k(y)|\,dy\\&\le \nonumber C\|\nabla b\|_{L^\infty}Mf_k(x).\end{align}  Therefore by the $L^p(\ell^2)({\Bbb R}^n)$-boundedness of $M$ (see \cite{FS}), we get the desired result. \qed

\begin{lemma}\label{CD1}\, Let $M_{s,\delta,j} \in C_0^\infty(\mathbb{R}^n)(0<\delta <
\infty)$  for any fixed $s,\,j \in \mathbb Z,$ and $T_{s,\delta,j}$ be the multiplier operator defined by
$\widehat{T_{s,\delta,j} f}(\xi)= M_{s,\delta,j}(\xi)\widehat{f}(\xi).$ Let $b \in
Lip(\mathbb R^n)$ and  $[b, T_{s,\delta,j}]$ be the commutator of $T_{s,\delta,j},$ which is defined by $$[b, T_{s,\delta,j}]f(x)=b(x)T_{s,\delta,j}f(x)-T_{s,\delta,j}(bf)(x).$$If for some positive constant $\beta$ and any fixed multi-index $\alpha$ with $|\alpha|=2,$
$$
\|M_{s,\delta,j}\|_{L^\infty} \leq
C2^{-j}\min \{2^{-(1+\beta) s},2^{ s}\}\min\{\delta^2,\delta^{-\beta}\},\,\,\,\,\|\partial^\alpha M_{s,\delta,j}\|_{L^\infty} \leq C2^{-j}2^{ s},
$$
then there exist some constants $0<\lambda,\gamma<1$ such that
$$\begin{array}{cl}
\|[b,T_{s,\delta,j}]f\|_{L^2} \leq C  2^{-j}\min \{2^{-\gamma s},2^{ s}\}\min\{\delta ^{2\lambda}, \delta
^{-\beta \lambda}\}\|b\|_{Lip} \|f\|_{L^2},
\end{array}$$where $C$ is independent of $s,\,\delta$ and $j$.\end{lemma}

{\it  Proof.}
  Taking a $C_0^\infty(\mathbb{R}^n)$ radial
function $\varphi$ with supp $\varphi\subset \{1/2\leq |x| \leq 2\}$ and
$\sum_{l \in \mathbb Z}\varphi(2^{-l}x)=1$ for any $|x|>0.$  Denote
$\varphi_0(x)=\sum_{l=-\infty}^0\varphi (2^{-l}x)$ and $\varphi_l(x) =
\varphi(2^{-l}x),$ for positive integer $l$. Let
$K_{s,\delta,j} (x)=M_{s,\delta,j}^{\vee}(x),$ the inverse Fourier transform of
$M_{s,\delta,j}.$ Splitting $K_{s,\delta,j}$ into
$$\begin{array}{cl}
K_{s,\delta,j}(x)=K_{s,\delta,j}(x)\varphi_0(x)+\dsum_{l=1}^\infty
K_{s,\delta,j}(x)\varphi_l(x)=:\dsum_{l=0}^\infty K_{s,\delta,j}^l(x).
\end{array}$$
Write $$\begin{array}{cl}
\widehat{K_{s,\delta,j} ^l}(x)=\dint_{\mathbb{R}^n} M_{s,\delta,j}
(x-y)\widehat{\varphi_l}(y)\,dy.\end{array}$$Since $\varphi$ is null in a neighbornhood of the origin and a Schwartz function, we have \begin{align}\label{eta}\int_{\mathbb{R}^n}\widehat{\varphi}(\eta)\eta^\vartheta\,d\eta=0\end{align} for any multi-index $\vartheta.$
Then expanding $M_{s,\delta,j}(x)$ into a Taylor series around $x$ and \eqref{eta} gives that
\begin{align}\label{eta1}
\|\widehat{K_{s,\delta,j} ^l}\|_{L^\infty}&\le\dsum_{|\alpha|=2}\|\partial^\alpha M_{s,\delta,j}\|_{L^\infty}\dint_{\mathbb{R}^n} |y|^{2}|\widehat{\varphi_{l}}(y)|\,dy\\&
 \leq \nonumber\dsum_{|\alpha|=2}\|\partial^\alpha M_{s,\delta,j}\|_{L^\infty}
\dint_{\mathbb{R}^n} |{2^{-l}}y|^{2}|\widehat{\varphi}(y)|\,dy \\&
 \leq\nonumber C2^{-j}2^{-2l}2^{ s}\dint_{\mathbb{R}^n}
 |y|^{2}|\widehat{\varphi}(y)|\,dy \\&
 \leq \nonumber C2^{-2l}2^{-j}2^{ s}.
\end{align}
On the other hand, by the Young inequality,
\begin{align}\label{eta2}
\|\widehat{K_{s,\delta,j}
^l}\|_{L^\infty}&=\|\widehat{K_{s,\delta,j}}\ast\widehat{\varphi_l}\|_{L^\infty}\\&\leq \nonumber
\|\widehat{K_{s,\delta,j}}\|_{L^\infty}\|\widehat{\varphi_l}\|_{L^1}\\&\leq\nonumber
C2^{-j} \min \{2^{-(1+\beta) s},2^{ s}\}\min\{\delta^2,\delta^{-\beta}\}.
\end{align}
Therefore, interpolating between \eqref{eta1} and \eqref{eta2}, for each $1/2<\theta<\frac{1+\beta}{2+\beta}$,
\begin{align*}
\|\widehat{K_{s,\delta,j} ^l}\|_{L^\infty} \le C2^{-2\theta
l}2^{- j}
\min\{\delta^{2(1-\theta)},\delta^{-(1-\theta)\beta}\}\min \{2^{\theta s-(1+\beta) (1-\theta)s},2^{s}\}.
\end{align*} Denote by $\gamma:=\theta-(1+\beta) (1-\theta)<0$ and $\lambda:=1-\theta>0,$ we get \begin{align}\label{eta3}
\|\widehat{K_{s,\delta,j} ^l}\|_{L^\infty} \le C2^{-2\theta
l}2^{- j}
\min\{\delta^{2\lambda},\delta^{-\lambda\beta}\}\min \{2^{\gamma s},2^{s}\}.
\end{align}
 Now we turn our attention to $[b,T_{s,\delta,j} ^l]$ the commutator of
the operator $T_{s,\delta,j} ^l.$ Decompose $\mathbb{R}^n$ into a grid
of non-overlapping cubes with side length $2^l.$ That is, $\mathbb{R}^n =
\cup _{d=-\infty}^\infty Q_d.$ Set $f_d = f\chi_{Q_d},$ then
$$
f(x) = \dsum_{d=-\infty}^\infty f_d(x),\qquad      a.e.\ \ x \in \mathbb{R}^n.
$$
 It is obvious
that supp $([b,T_{s,\delta,j}^l]f_d) \subset 2nQ_d$ and that the supports
of $\{[b,T_{s,\delta,j} ^l]f_d \}_{d=-\infty}^{+\infty}$ have bounded
overlaps. So we have the following almost orthogonality property
$$\begin{array}{cl}
\|[b,T_{s,\delta,j} ^l]f\|_{L^2}^2 \leq C\dsum_{d=-\infty}^\infty
\|[b,T_{s,\delta,j} ^l] f_d\|_{L^2}^2.
\end{array}$$
Thus, we may assume that supp $f \subset Q$ for some cube with side
length $2^l.$ Choose $\psi \in C_0^\infty(\mathbb{R}^n)$ with $0 \leq
\psi\leq 1$, supp $\psi
\subset 100nQ$ and $\psi=1,$ when $x \in 30nQ$.  Set $\widetilde{Q}=200nQ$ and $\widetilde{b} =
(b(x)-b_{\widetilde{Q}})\psi(x),$ we can get
$$\begin{array}{cl}
\|[b,T_{s,\delta,j}]f\|_{L^2}\le\dsum_{l\ge 0}\|[b,T_{s,\delta,j} ^l]f\|_{L^2}\leq \dsum_{l\ge 0}\|\widetilde{b}T_{s,\delta,j}^lf\|_{L^2}+\dsum_{l\ge 0}\|T_{s,\delta,j}^l(\widetilde{b}f)\|_{L^2}.
\end{array}$$
By \eqref{eta3} with  $\theta>1/2$ and $\|\widetilde{b}\|_{L^\infty}\le 2^l\|b\|_{Lip}$, we have
$$\begin{array}{cl}
\dsum_{l\ge 0}\|\widetilde{b}T_{s,\delta,j}^lf\|_{L^2} &\leq
\dsum_{l\ge 0}\|\widetilde{b}\|_{L^\infty}\|T_{s,\delta,j}^lf\|_{L^2}\\&\leq
C\dsum_{l\ge 0}2^{(1-2\theta) l}2^{- j}\|b\|_{Lip}
\min\{\delta^{2\lambda},\delta^{-\lambda\beta}\}\min \{2^{\gamma s},2^{s}\}\|f\|_{L^2}\\&\leq
C2^{- j}\|b\|_{Lip}
\min\{\delta^{2\lambda},\delta^{-\lambda\beta}\}\min \{2^{\gamma s},2^{s}\}\|f\|_{L^2}.
\end{array}$$
 Similarly, we can get
$$\begin{array}{cl}
\dsum_{l\ge 0}\|T_{s,\delta,j}^l(\widetilde{b}f)\|_{L^2}&
\leq C2^{- j}\|b\|_{Lip}\min\{\delta^{2\lambda},\,\delta^{-\beta\lambda}\}\min \{2^{\gamma s},2^{s}\}\|f\|_{L^2}.
\end{array}$$
Thus
$$\begin{array}{cl}
\|[b,T_{s,\delta,j}]f\|_{L^2}&\le C2^{- j}\|b\|_{Lip}\min\{\delta^{2\lambda},\,\delta^{-\beta\lambda}\}\min \{2^{\gamma s},2^{s}\}\|f\|_{L^2},
\end{array}$$ where $C$ is independent of $\delta,\,s$ and $j$.
\qed

\section{Proof of Theorem \ref{thm:H} (I) }

As we have stated in the introduction, to prove Theorem \ref{thm:H} it suffices to show \eqref{dyadic jump singular} and \eqref{short variation singular}.
In this section, we give the proof of \eqref{dyadic jump singular}.  For
$j\in\mathbb Z$,  let $\nu_j(x)=\frac{\Omega(x)}{|x|^{n+1}}\chi_{\{2^j\le |x|<2^{j+1}\}}(x)$, then
$$
\nu_j\ast f(x)=\int_{2^j\le |y|<2^{j+1}}\frac{\Omega(y)}{|y|^{n+1}}f(x-y)dy.
$$ Denote by $$T^1f(x)=\hbox{p.v.}\dint_{{\Bbb R}^n}\frac{\Omega(y)}{|y|^{n+1}}f(x-y)dy$$ and for $k\in\Bbb Z$
$$T^1_{2^k}f(x)=\dint_{|x-y|>2^k}\frac{\Omega(y)}{|y|^{n+1}}f(x-y)dy.$$
Let $\phi\in \mathscr{S}(\mathbb R^n)$ be a radial function such that $\hat{\phi}(\xi)=1$ for
$|\xi|\le2$ and $\hat{\phi}(\xi)=0$ for $|\xi|>4$. We have the following decomposition
\begin{align*}
T^1_{2^k}f&=\phi_k\ast T^1f+\sum_{s\ge0}(\delta_0-\phi_k)\ast\nu_{k+s}\ast f-\phi_k\ast\sum_{s<0}\nu_{k+s}\ast
f,
\end{align*}
where $\phi_k$ satisfies $\widehat{\phi_k}(\xi)=\hat{\phi}(2^k\xi)$, $\delta_0$ is the Dirac measure at 0.
Then \begin{align*}
\mathcal{C}_{2^k}f&=[b,\phi_k\ast T^1]f+[b,\sum_{s\ge0}(\delta_0-\phi_k)\ast\nu_{k+s}] f-[b,\phi_k\ast\sum_{s<0}\nu_{k+s}]
f\\
&:=\mathcal{C}^1_{k}f+\mathcal{C}^2_{k}f-\mathcal{C}_{k}^3f.
\end{align*}
Let $\mathscr C^if$ denote the family $\{\mathcal{C}^i_{k}f\}_{k\in\mathbb Z}$  for
$i=1,2,3$. Obviously, to show \eqref{dyadic jump singular} it suffices to prove the following inequalities:
 \begin{align}\label{lnti}
\|\lambda\sqrt{N_\lambda(\mathscr C^if)}\|_{L^p}\le C\|\nabla b\|_{L^\infty}\|f\|_{L^p},\ \ 1<p<\infty,\ \ i=1,2,3,
\end{align}uniformly in $\lambda>0.$
\bigskip

\noindent {\bf Estimation of \eqref{lnti} for $i=1$.} For $k\in \Bbb Z,$ denote by $\Phi_kf(x)=\phi_k\ast f(x)$ and
  write
 $$\mathcal{C}^1_{k}f=[b,\Phi_k] T^1f+\Phi_k [b,T^1]f.$$
Combining Theorem 1.1 in \cite{JSW08} and the $L^p\,(1<p<\infty)$-boundedness of $[b,T^1]$ with bounds $C\|\Omega\|_{L({\log^+L})(\mathbf S^{n-1})}\|\nabla b \|_{L^\infty}$ (see \cite{C1}), we can get the following estimate easily for $1<p<\infty$
\begin{align}\label{phik1}
\big\|\lambda\sqrt{N_\lambda(\{\Phi_k [b, T^1]f\}_k)}\big\|_{L^p}\le C\| [b, T^1]f\|_{L^p}\le C\|\Omega\|_{L({\log^+L})(\mathbf S^{n-1})}\|\nabla b\|_{L^\infty}\|f\|_{L^p}.
\end{align}
 Using $\sum_{j=1}^nR_j^2=-\mathcal{I}$ (identity operator) and $R_j=\partial_j I_1,\, j=1,\dots, n,$ to get \begin{align*}[b,\Phi_k] T^1f=-[b,\Phi_k]\sum_{j=1}^nR_j^2 T^1f=-\sum_{j=1}^n[b,\Phi_k]\partial_j( R_j I_1 T^1f),\end{align*} where $R_j$ is the $j$-th Riesz transform and $I_1$ is the Riesz potential operator of order $1.$
Then by Theorem \ref{Phi1} and $\|R_jf\|_{L^p}\le C\|f\|_{L^p}$ for $1<p<\infty,$ $j=1,\dots,n,$ we  get
\begin{align}\label{phik2}
\big\|\lambda\sqrt{N_\lambda(\{[b,\Phi_k] T^1f\}_k)}\big\|_{L^p}&\le \dsum_{j=1}^n\Big\|\lambda\sqrt{N_\lambda(\{[b,\Phi_k]\partial_j (R_j I_1 T^1f)\}_k)}\Big\|_{L^p}\\&\le \nonumber C\|\nabla b\|_{L^\infty}\dsum_{j=1}^n\| R_j T^1I_1 f\|_{L^p}\\&\le \nonumber C\|\Omega\|_{L({\log^+L})(\mathbf S^{n-1})}\|\nabla b\|_{L^\infty}\| (-\Delta)^{1/2} I_1 f\|_{L^p}\\&= \nonumber C\|\Omega\|_{L({\log^+L})(\mathbf S^{n-1})}\|\nabla b\|_{L^\infty}\|f\|_{L^p},
\end{align}
where in the above inequality, we have used that $\|T^1f\|_{L^p}\le C\|\Omega\|_{L({\log^+L})(\mathbf S^{n-1})}\|(-\Delta)^{1/2}f\|_{L^p}$ for $1<p<\infty$ (see \cite{CFY}) and $(-\Delta)^{1/2} I_1=\mathcal{I}$.
Together \eqref{phik1} with \eqref{phik2}, we get  for $1<p<\infty,$\begin{align*}
\big\|\lambda\sqrt{N_\lambda(\mathscr C^1f)}\big\|_{L^p}&\le\big\|\lambda\sqrt{N_\lambda(\{\Phi_k [b, T^1]f\}_k)}\big\|_{L^p}+\big\|\lambda\sqrt{N_\lambda(\{[b,\Phi_k] T^1f\}_k)}\big\|_{L^p}\\&\le C\|\Omega\|_{L({\log^+L})(\mathbf S^{n-1})}\|\nabla b\|_{L^\infty}\|f\|_{L^p}
\end{align*} uniformly in $\lambda>0.$

\noindent {\bf Estimation of (\ref{lnti}) for $i=2$.}
Let $E_0=\{ x^\prime\in \mathbf S^{n-1}:
|\Omega(x^\prime)|<2\}$
 and $E_{m}=\{ x^\prime\in \mathbf S^{n-1}:
 2^{m}\le|\Omega(x^\prime)|<2^{m+1}\}$
  for positive integer $m$.  For $m\ge 0,$ let
$$\Omega_m(y')=\Omega(y')\chi_{E_m}(y')-\frac{1}
{|\mathbf S^{n-1}|}\int_{E_m}\Omega(x' ) \,d\sigma(x').$$ Since $\Omega$
satisfies \eqref{can of O}, then
$$\int_{\mathbf S^{n-1}}\Omega_m(y')\,d\sigma(y')=0 \
 \ \hbox{for}\ \  m\ge 0$$
and $\Omega(y')=\sum_{m\ge 0}\Omega_m(y').$ Set
   $\nu_{j,m}(x)=\frac{\Omega_m(x)}
   {|x|^{n+1}}\chi_{\{2^j\le |x|
   < 2^{j+1}\}}(x),$ then $\nu_{j}(x)=\sum_{m\ge 0}\nu_{j,m}(x).$
Thus, by the  fact $\ell^2$ embeds into $\ell^{2,\infty}$ and the Minkowski inequality, we get
\begin{align}\label{C2}
\lambda\sqrt{N_\lambda(\mathscr C^2f)(x)}&\le \dsum_{s\ge0}\Big(\dsum_{k\in \Bbb Z}\Big|[b,(\delta_0-\phi_k)\ast\nu_{k+s}]
f(x)\Big|^2\Big)^{1/2}\\
&\le \nonumber\dsum_{s\ge0}\dsum_{m\ge 0}\Big(\dsum_{k\in \Bbb Z}\Big|[b,(\delta_0-\phi_k)\ast\nu_{k+s,m}]
f(x)\Big|^2\Big)^{1/2}.
\end{align} Denote by $F_{s,k,m}f(x):=(\delta_0-\phi_k)\ast\nu_{k+s,m}\ast f(x).$  Let $\varphi\in
C_0^\infty({\Bbb R}^n)$ be a radial function such
  that
 $0\le \varphi\le 1,$ $\hbox{supp}\ \varphi \subset\{1/2\le |\xi|\le
2\}$ and $\sum_{l\in \Bbb Z}\varphi^2(2^{-l}\xi)=1$ for $|\xi|\neq 0.$
Define the multiplier $\Delta_l$ by
$ \widehat{\Delta_l f}(\xi)=\varphi(2^{-l}\xi)\widehat{f}(\xi).$
It is clear that
\begin{align*}[b, (\delta_0-\phi_k)\ast\nu_{k+s,m} ]f(x)=[b, F_{s,k,m} ]f(x)=\sum_{l\in \Bbb Z}[b,F_{s,k,m}\Delta_{l-k}^2]f(x).\end{align*}
Then by the Minkowski inequality, we get for $1<p<\infty,$
\begin{align}\label{C3}
\Big\|\Big(\dsum_{k\in \Bbb Z}\Big|[b,(\delta_0-\phi_k)\ast\nu_{k+s,m}]
f\Big|^2\Big)^{1/2}\Big\|_{L^p}
&\le\dsum_{l\in \Bbb Z}\bigg\|
 \bigg(\dsum_{k\in \Bbb Z}
 |[b, F_{s,k,m}\Delta_{l-k}^2 ]f|^2\bigg)^{1/2}\bigg\|_{L^p}\\
 &:\nonumber=\dsum_{l\in \Bbb Z}\|G_{s,m;b}^{l}f\|_{L^p}.
\end{align}
If we can prove the following two inequalities: for some $0<\beta<1$ and $0<\theta<1,$\begin{align}\label{2bound2}
 \|G_{s,m;b}^lf\|_{L^2}\le C2^{- \beta s}2^{-\theta|l|}\|\Omega_m\|_{L^\infty(\mathbf S^{n-1})}\|\nabla b\|_{L^\infty}\|f\|_{L^2}
\end{align} and \begin{align}\label{Gp}\|G_{ s,m; b}^lf\|_{L^p}&\le
 C\|\Omega_m\|_{L^1(\mathbf S^{n-1})}\|\nabla b\|_{L^\infty}\|f\|_{L^p}\,\,\,\,\,\hbox{for}\,\,\,\,\,\,1<p<\infty,
\end{align} we may get (\ref{lnti}) for $i=2$.
In fact, interpolating between \eqref{2bound2} and \eqref{Gp},  we get for $0<\theta_0, \beta_0<1$,
\begin{align}\label{Gps}
\|G_{ s,m; b}^lf\|_{L^p}\le C 2^{-\beta_0 s}2^{- \theta_0|l|}
\|\nabla b\|_{L^\infty}\|\Omega_m\|_{L^\infty(\mathbf S^{n-1})}\|f\|_{L^p},\,\,\,1<p<\infty.
\end{align}
Taking a large positive integer $N$, such that
$N>\max\{2\theta_0^{-1}, 2\beta_0^{-1}\}$,  we have for $1<p<\infty,$
\begin{align*}
\|\lambda\sqrt{N_\lambda(\mathscr C^{2}f)}\|_{L^p}&\le\dsum_{m\ge 0 }\dsum_{0\le s<Nm}\dsum_{|l|<Nm}\|G_{ s,m; b}^lf\|_{L^p}+\dsum_{m\ge 0 }\dsum_{0\le s<Nm}\dsum_{|l|\ge Nm}\|G_{ s,m; b}^lf\|_{L^p}\\&+\dsum_{m\ge 0 }\dsum_{s>Nm}\dsum_{|l|\ge 0}\|G_{ s,m; b}^lf\|_{L^p}.
\end{align*}
By \eqref{Gp}, we get for $1<p<\infty,$
\begin{align*}
\dsum_{m\ge 0 }\dsum_{0\le s<Nm}\dsum_{|l|<Nm}\|G_{ s,m; b}^lf\|_{L^p}&\le
C\|\nabla b\|_{L^\infty}\dsum_{m\ge  0}\dsum_{0\le s<Nm}\dsum_{0\le|l|< Nm}2^m
\sigma(E_{m}) \|f\|_{L^p}\\
&\le C\|\nabla b\|_{L^\infty}\dsum_{m\ge
0}m^22^m \sigma(E_{m}) \|f\|_{L^p}\\
&\le C\|\Omega\|_{L(\log^+\!\!L)^2(\mathbf S^{n-1})}\|\nabla b\|_{L^\infty}\|f\|_{L^p}.
\end{align*}
Applying \eqref{Gps}, we get for $1<p<\infty,$
\begin{align*}
\dsum_{m\ge 0 }\dsum_{0\le s<Nm}\dsum_{|l|\ge Nm}\|G_{ s,m; b}^lf\|_{L^p} &\leq C\|\nabla b\|_{L^\infty}\dsum_{m\ge  0 }2^m\dsum_{0\le s<Nm}2^{-\beta_0 s}\dsum_{|l|>Nm}2^{-\theta_0|l|}\|f\|_{L^p}\\&\le C\|\nabla b\|_{L^\infty}\dsum_{m\ge  0 }m2^{(1-\theta_0N)m}\|f\|_{L^p}\\ &\leq C\|\nabla b\|_{L^\infty}\|f\|_{L^p}.
\end{align*}
Applying \eqref{Gps} again, we get for $1<p<\infty,$
\begin{align*}
\dsum_{m\ge 0 }\dsum_{s>Nm}\dsum_{|l|\ge 0}\|G_{ s,m; b}^lf\|_{L^p}&\leq \dsum_{m\ge 0 }\dsum_{s>Nm}\dsum_{0\le |l|<Nm}\|G_{ s,m; b}^lf\|_{L^p}+\dsum_{m\ge 0 }\dsum_{s>Nm}\dsum_{|l|\ge Nm}\|G_{ s,m; b}^lf\|_{L^p}\\&\le C\|\nabla b\|_{L^\infty}\dsum_{m\ge  0 }2^m\dsum_{s>Nm}2^{-\beta_0s}\bigg(\dsum_{|l|<Nm}+\dsum_{|l|\ge Nm}2^{-\theta_0|l|}\bigg)\|f\|_{L^p}\\&\le C\|\nabla b\|_{L^\infty}\dsum_{m\ge  0 }(m2^{(1-\beta_0N)m}+2^{(1-\beta_0N-\theta_0N)m})\|f\|_{L^p}\\ &\leq C\|\nabla b\|_{L^\infty}\|f\|_{L^p}.
\end{align*}
Combining above three estimates, we have for $1<p<\infty$
\begin{align*}\|\lambda\sqrt{N_\lambda( \mathscr C^2f)}\|_{L^p}\le
C(1+\|\Omega\|_{L(\log^+\!\!L)^2(\mathbf S^{n-1})})\|\nabla b\|_{L^\infty}\|f\|_{L^p}.
\end{align*}
We therefore finish the  estimate of (\ref{lnti}) for $i=2$.

Now we are going to give the proof of  \eqref{2bound2} and \eqref{Gp}.
We first prove a rapid decay estimate of $\|G_{s,m;b}^{l}f\|_{L^2}$ for $l\in \Bbb Z$ and $s\in \Bbb N$.  Set $F_{s,k,m}^lf(x):=F_{s,k,m}\Delta_{l-k}f(x).$ Write
$$
[b, F_{s,k,m}\Delta_{l-k}^2 ]f=
 F_{s,k,m}^l [b, \Delta_{l-k}]f+[b, F_{s,k,m}^l ]\Delta_{l-k}f.
$$
Therefore
\begin{align*}
\|G_{s,m;b}^{l}f\|_{L^2}&\le
 \bigg\|
 \bigg(\dsum_{k\in \Bbb Z}
 |F_{s,k,m}^l [b, \Delta_{l-k}]f|^2\bigg)^{1/2}\bigg\|_{L^2}+
 \bigg\|
 \bigg(\dsum_{k\in \Bbb Z}
 |[b,F_{s,k,m}^l]\Delta_{l-k}f|^2\bigg)^{1/2}\bigg\|_{L^2}\\
 &:= I+II.
\end{align*}
 Set   $$M_{s,k,m}(\xi)=(1-\widehat{\phi_k}(\xi))\widehat{\nu_{k+s,m}}(\xi),\,\,\,\,\,M_{s,k,m}^l(\xi)=M_{s,k,m}(\xi)\varphi(2^{k-l}\xi).$$ Then write $F_{s,k,m}$  and $F_{s,k,m}^l$, respectively by $$\widehat{F_{s,k,m}f}(\xi)=M_{s,k,m}(\xi)\widehat{f}(\xi) \,\,\hbox{and}\,\, \widehat{F_{s,k,m}^lf}(\xi)=M_{s,k,m}^l(\xi)\widehat{f}(\xi).$$
 Since  $supp\ (1-\widehat{\phi_k})\widehat{\nu_{k+s,m}} \subset \{\xi :|2^k\xi|>1/2\}$,
by a well-known Fourier transform estimate of Duoandikoetxea and Rubio de Francia
 (See \cite{DR86}, p.551-552), it is easy to show that there exists some  $\nu\in(0,1)$ such that
\begin{align}\label{MF}|M_{s,k,m}(\xi)|\le C 2^{-k} 2^{ -(\nu+1) s}\min\{|2^k \xi|^2, |2^k \xi|^{-\nu}\}\|\Omega_m\|_{L^\infty(\mathbf S^{n-1})},\,\,\,s\ge 0.\end{align}
From this and the Plancherel theorem imply  the following estimate
\begin{align}\label{2bound1}
\|F_{s,k,m}^lf\|_{L^2}\le C2^{ -k} 2^{-(\nu+1) s} \min\{  2^{2l}, 2^{-\nu l}\}\|\Omega_m\|_{L^\infty(\mathbf S^{n-1})}\|f\|_{L^2},\ \ \ \ \ \text{for} \ l\in \Bbb Z.
\end{align}
Then apply \eqref{2bound1} and Lemma \ref{Deltacom}, we have
\begin{align*}
I&\le C 2^{-(\nu+1) s}\min\{ 2^{-(\nu+1) l}, 2^{l}\}\|\Omega_m\|_{L^\infty(\mathbf S^{n-1})} \bigg\|
 \bigg(\dsum_{k\in \Bbb Z}
 |2^{l -k}[b,\Delta_{l-k}]f|^2\bigg)^{1/2}\bigg\|_{L^2}\\
 &\le C2^{-(\nu+1) s}\min\{ 2^{-(\nu+1) l}, 2^{l}\}\|\Omega_m\|_{L^\infty(\mathbf S^{n-1})} \|\nabla b\|_{L^\infty}\|f\|_{L^2}.
\end{align*}

To proceed with the estimate of $II$, we define multiplier $\widetilde{F}_{s,k,m}^l$ by $\widehat{\widetilde{F}_{s,k,m}^l f}(\xi)=M_{s,k,m}^l(2^{-k}\xi)\widehat{f}(\xi).$ As a result of \eqref{MF}, we have the following estimate
\begin{align}\label{m}
|M_{s,k,m}^l(2^{-k}\xi)|\le C2^{-k} 2^{ -(\nu+1) s} \min\{2^{2l}, 2^{-\nu l}\}\|\Omega_m\|_{L^\infty(\mathbf S^{n-1})}.
\end{align}
On the other hand, by the trivial computation,  we have for any fixed multi-index $\eta$,\begin{align}\label{Fourier}
|\partial^\eta\widehat{\nu_{k+s,m}}(\xi)|&\le C2^{(k+s)(|\eta|-1)}\|\Omega_m\|_{L^1(\mathbf S^{n-1})}.
\end{align} Then we have for  any fixed multi-index $\alpha$ with $|\alpha|=2,$\begin{align}\label{Fourier1}|\partial^\alpha (M_{s,k,m}^l(2^{-k}\xi))|
&=|\partial^\alpha\big(\widehat{\nu_{k+s,m}}(2^{-k}\xi)(1-\phi(\xi))\varphi(2^{-l}\xi)\big)|\\&\nonumber=|\dsum_{\eta}C_{\eta_1}^{\alpha_1}\dots
C_{\eta_n}^{\alpha_n}\partial^\eta(\widehat{\nu_{k+s,m}}(2^{-k}\xi) )\partial^{\alpha-\eta}[(1-\phi(\xi))\varphi(2^{-l}\xi)]|\\&\le \nonumber C2^{-k}2^s\|\Omega_m\|_{L^1(\mathbf S^{n-1})},\end{align} where the sum is taken over all multi-indices $\eta$ with $0\le \eta_j\le \alpha_j$ for $1\le j\le n.$
Via Lemma \ref{CD1}  to \eqref{m} and \eqref{Fourier1} with $\delta=2^l$ and $j=k$ says that there exist  constants $\vartheta\in (0,1)$ and $\gamma\in (0,1)$ such that
$$ \|[b,\widetilde{F}_{s,k,m}^l] f\|_{L^2} \le C 2^{-k}2^{-\vartheta s}2^{-\gamma|l|}\|\Omega_m\|_{L^\infty(\mathbf S^{n-1})}\|b\|_{Lip}\|f\|_{L^2},\ \ \ \ \ for \ \ \ l\in \Bbb Z \ \ and\ \ s\ge 0.$$
Further, by $\|b(2^k\cdot)\|_{Lip}=2^k\|b\|_{Lip},$ we have
\begin{align}\label{2bound}\|[b,{F}_{s,k,m}^l] f\|_{L^2}\le C2^{-\vartheta s}2^{-\gamma|l|}\|\Omega_m\|_{L^\infty(\mathbf S^{n-1})}\|b\|_{Lip}\|f\|_{L^2},\ \ \ \ \ for \ \ \ l\in \Bbb Z \ \ and\ \ s\ge 0.\end{align}
 Then by \eqref{2bound} and  Littlewood-Paley theory, we get
\begin{align*}
II&\le C2^{-\vartheta s}2^{-\gamma |l|}\|\Omega_m\|_{L^\infty(\mathbf S^{n-1})}\|\nabla b\|_{L^\infty} \bigg\|
 \bigg(\dsum_{k\in \Bbb Z}
 |\Delta_{l-k}f|^2\bigg)^{1/2}\bigg\|_{L^2}\\
 &\le C2^{-\vartheta s}2^{-\gamma |l|}\|\Omega_m\|_{L^\infty(\mathbf S^{n-1})}\|\nabla b\|_{L^\infty}\|f\|_{L^2}.
\end{align*}
Combining the estimates of  $I$  with  $II$, we establish the proof of  \eqref{2bound2}.

\par
Now we give the $L^p(1<p<\infty)$ estimate of $G_{s,m;b}^lf$ for $l\in \Bbb Z$ and $s\ge 0$.
We write
\begin{align*}
[b, F_{s,k,m}\Delta_{l-k}^2 ]f(x)=
[b, F_{s,k,m} ]\Delta_{l-k}^2f+ F_{s,k,m} [b, \Delta_{l-k}^2]f.
\end{align*}
 By the Minkowski inequality, we get  for $1<p<\infty$
$$\begin{array}{ll}\|G_{ s,m; b}^lf\|_{L^p}&\le
 \bigg\|
 \bigg(\dsum_{k\in \Bbb Z}
 |[b,F_{s,k,m}]\Delta_{l-k}^2f|^2\bigg)^{1/2}\bigg\|_{L^p}\\&\quad+
 \bigg\|
 \bigg(\dsum_{k\in \Bbb Z}
 |F_{s,k,m} [b, \Delta_{l-k}^2]f|^2\bigg)^{1/2}\bigg\|_{L^p}.
\end{array}$$We estimate each term separately. Firstly, we estimate $\|
 (\sum_{k\in \Bbb Z}
 |[b,F_{s,k,m}]\Delta_{l-k}^2f|^2)^{1/2}\|_{L^p}$ for $1<p<\infty.$ Recall that $F_{s,k,m}f=(\delta_0-\phi_k)\ast\nu_{k+s,m}\ast f$ and $\Phi_kf=\phi_k\ast f.$
 Denote by $T_{j,m}f=\nu_{j,m}\ast f$ for $j\in \Bbb Z$  and $m\ge 0.$  Write
\begin{align*}[b,F_{s, k, m}] f=[b,T_{k+s, m}]\Phi_{k} f+{
T}_{k+s,m}[b,\Phi_{k}] f-[b, T_{k+s, m}]f.
\end{align*}
It is well known that for any $f\in L^p(\mathbb R^n),$
\begin{align*}|[b,T_{k+s,m}]f(x)|
&\le
C\|b\|_{Lip}M_{\Omega_m}f(x).\end{align*}
From this and $M_{\Omega_m}$ is bounded on $L^p(\ell^2)({\Bbb R}^n)$ for $1<p<\infty$ with bounds $\|\Omega_m\|_{L^1(\mathbf S^{n-1})}$ (see Lemma 2.3 in \cite{CD})  we  get for $ 1<p<\infty$,\begin{align}\label{Tsk1}\bigg\|\bigg(\dsum_{k\in \mathbb Z}|[b,T_{k+s,m}]f_k|^2\bigg)^{1/2}\bigg\|_{L^p}&\le
C\|\Omega_m\|_{L^1(\mathbf S^{n-1})}\|\nabla b\|_{L^\infty}\bigg\|\bigg(\dsum_{k\in \mathbb Z}|f_k|^2\bigg)^{1/2}\bigg\|_{L^p}.\end{align}
By Lemma \ref{Tsk} and Lemma \ref{phicom},  we get for $ 1<p<\infty,$ \begin{align}\label{Tsk3}\bigg\|\bigg(\dsum_{k\in \mathbb Z}|{
T}_{k+s,m}[b,\Phi_{k}]f_k|^2\bigg)^{1/2}\bigg\|_{L^p}&\le
C\|\Omega_m\|_{L^1(\mathbf S^{n-1})}\bigg\|\bigg(\dsum_{k\in \mathbb Z}|\nabla[b,\Phi_{k}]f_k|^2\bigg)^{1/2}\bigg\|_{L^p}\\&\le\nonumber
C\|\nabla b\|_{L^\infty}\|\Omega_m\|_{L^1(\mathbf S^{n-1})}\bigg\|\bigg(\dsum_{k\in \mathbb Z}|f_k|^2\bigg)^{1/2}\bigg\|_{L^p}.\end{align}
Together \eqref{Tsk1}-\eqref{Tsk3} with  the $L^p(\ell^2)({\Bbb R}^n)\,(1<p<\infty)$ boundedness of $\{\Phi_{k}\}$ and Littlewood-Paley theory, we get for $1<p<\infty,$
 \begin{align}\label{Lp}
  \bigg\|
 \bigg(\dsum_{k\in \Bbb Z}
 |[b,F_{s,k,m}]\Delta_{l-k}^2f|^2\bigg)^{1/2}\bigg\|_{L^p}&\le
 C\|{\Omega_m}\|_{L^1(\mathbf S^{n-1})}\|\nabla b\|_{L^\infty} \bigg\|
 \bigg(\dsum_{k\in \Bbb Z}
 |\Delta_{l-k}^2f|^2\bigg)^{1/2}\bigg\|_{L^p}\\&\le\nonumber
 C\|{\Omega_m}\|_{L^1(\mathbf S^{n-1})}\|\nabla b\|_{L^\infty} \|f
\|_{L^p}.
\end{align}
Secondly, we estimate $\|
(\sum_{k\in \Bbb Z}
 |F_{s,k,m} [b, \Delta_{l-k}^2]f|^2)^{1/2}\|_{L^p}$ for $1<p<\infty.$
If the following inequality holds   \begin{align}\label{Lp2}\bigg\|
 \bigg(\dsum_{k\in \Bbb Z}
 | \nabla[b, \Delta_{l-k}^2]f|^2\bigg)^{1/2}\bigg\|_{L^p}
&\le
 C\|\nabla b\|_{L^\infty}\|f\|_{L^p},\,\,\,\,\,1<p<\infty.
\end{align}
Then apply $L^p(\ell^2)({\Bbb R}^n)\,(1<p<\infty)$ boundedness of $\{\Phi_{k}\}$ and Lemma \ref{Tsk}, we can get
\begin{align}\label{Lpf}
 \bigg\|
 \bigg(\dsum_{k\in \Bbb Z}
 |F_{s,k,m}[b, \Delta_{l-k}^2] f|^2\bigg)^{1/2}\bigg\|_{L^p}&\le
 C\|\Omega_m\|_{L^1(\mathbf S^{n-1})}\bigg\|
 \bigg(\dsum_{k\in \Bbb Z}
 |\nabla[b, \Delta_{l-k}^2]  f_k|^2\bigg)^{1/2}\bigg\|_{L^p}\\&\le\nonumber
 C\|\Omega_m\|_{L^1(\mathbf S^{n-1})}\|\nabla b\|_{L^\infty}\|f\|_{L^p}.
\end{align}
Combining the estimates of \eqref{Lp} and \eqref{Lpf}, we get for $1<p<\infty,$
\begin{align}\nonumber\|G_{ s,m; b}^lf\|_{L^p}&\le
 C\|\Omega_m\|_{L^1(\mathbf S^{n-1})}\|\nabla b\|_{L^\infty}\|f\|_{L^p}.
\end{align}This gives \eqref{Gp}.
Now we prove \eqref{Lp2}. Since  $\nabla\Delta_{l-k}^2f(x)=2^{l-k}\widetilde{\Delta}_{l-k}(x)$ for  a.e. $x\in {\Bbb R}^n$, where $\widetilde{{\Delta}}_j$ is the Littlewood-paley operator given on the
transform by multiplication with the function
$(2^{-j}\xi)\varphi^2(2^{-j}\xi)$ for $j\in \Bbb Z$. Then by $\nabla[b, \Delta_{l-k}^2]f=[b, \nabla\Delta_{l-k}^2]f-[b, \nabla]\Delta_{l-k}^2f$ and the Minkowski inequality, we get for $1<p<\infty,$
\begin{align*}\nonumber\bigg\|
 \bigg(\dsum_{k\in \Bbb Z}
 | \nabla[b, \Delta_{l-k}^2]f|^2\bigg)^{1/2}\bigg\|_{L^p}&\le
 \bigg\|
 \bigg(\dsum_{k\in \Bbb Z}
 | [b, 2^{l-k}\widetilde{\Delta}_{l-k}]f|^2\bigg)^{1/2}\bigg\|_{L^p}+ \bigg\|
 \bigg(\dsum_{k\in \Bbb Z}
 | [b, \nabla]\Delta_{l-k}^2f|^2\bigg)^{1/2}\bigg\|_{L^p}.
\end{align*}
 By Lemma \ref{Deltacom}, $[b, \nabla]f=-(\nabla b)f$ and Littlewood-Paley theory, we get for $1<p<\infty,$
 \begin{align*}\nonumber\bigg\|
 \bigg(\dsum_{k\in \Bbb Z}
 | \nabla[b, \Delta_{l-k}^2]f|^2\bigg)^{1/2}\bigg\|_{L^p}
&\le
 C\|\nabla b\|_{L^\infty}\|f\|_{L^p}.
\end{align*}
This gives \eqref{Lp2}.

\noindent {\bf Estimation of (\ref{lnti}) for $i=3$.}
 We have the following pointwise estimate
\begin{align}\label{H}
\lambda\sqrt{N_\lambda(\mathscr C^3f)(x)}
&\le\dsum_{s<0}\big(\dsum_{k\in \Bbb Z}\big|[b, \phi_k\ast\nu_{k+s}]
f(x)\big|^2\big)^{1/2}.
\end{align}
The proofs are essentially similar to  the proof of  (\ref{lnti}) for $i=2$. More precisely,  we need to give the estimates on  the left hand side of (\ref{2bound2})-(\ref{Gp}) with replacing $(\delta_0-\phi_k)\ast\nu_{k+s}$ by $\phi_k\ast\nu_{k+s}$. Since  $supp\ \widehat{\phi_k}\widehat{\nu_{k+s}} \subset \{\xi :|2^k\xi|<1\}$ and $\Omega$ satisfies \eqref{mean zero}, then it is easy to see that
$$|\widehat{\phi_k\nu_{k+s}} (\xi)|\le C2^{-k} 2^{s}\|\Omega\|_{L^1(\mathbf S^{n-1})} \min\{|2^k \xi|^2, |2^k \xi|^{-1}\}$$  and for any fixed multi-index $\eta$ with $|\eta|\le 2,$ \begin{align*}
|\partial^\eta\widehat{\nu_{k+s}} (\xi)|&\le C2^{(k+s)(|\eta|-1)}\|\Omega\|_{L^1(\mathbf S^{n-1})}|2^{k+s}\xi|^{2-|\eta|}\le C2^{k(|\eta|-1)}2^s\|\Omega\|_{L^1(\mathbf S^{n-1})}|2^{k}\xi|^{2-|\eta|}.
\end{align*}
 Set   $$R_{s,k}(\xi)=\widehat{\phi_k}(\xi)\widehat{\nu_{k+s}}(\xi),\,\,\,\,\,R_{s,k}^l(\xi)=R_{s,k}(\xi)\varphi(2^{k-l}\xi).$$
Using the two above inequalities, we have the following estimate
\begin{align}\label{m2}
|R_{s,k}^l(2^{-k}\xi)|\le C2^{-k} 2^{ s} \min\{2^{2l}, 2^{- l}\}\|\Omega\|_{L^1(\mathbf S^{n-1})}.
\end{align} and
 for  any fixed multi-index $\alpha$ with $|\alpha|=2,$\begin{align}\label{Fourier2}|\partial^\alpha (R_{s,k}^l(2^{-k}\xi))|&\le  C2^{-k}2^s\|\Omega\|_{L^1(\mathbf S^{n-1})}.\end{align}
Then apply Lemma \ref{CD1} and the same arguments of the proofs of  (\ref{lnti}) for $i=2$, then  the right hand side of (\ref{2bound2}) is controlled by $ C2^{  s}2^{-\theta |l|}\|\Omega\|_{L^1(\mathbf S^{n-1})}\|\nabla b\|_{L^\infty}\|f\|_{L^2}$ for some  $\theta>0.$
It is also  easy to get the same estimates in the right hand side of (\ref{Gp}) by using \eqref{Tsk1}-\eqref{Tsk3} and Lemma  \ref{Tsk}.
Then we get for $1<p<\infty$
$$\|\lambda\sqrt{N_\lambda(\mathscr C^{3}f)}\|_{L^p}\le
C\|\Omega\|_{L^1(\mathbf S^{n-1})}\|\nabla b\|_{L^\infty}\|f\|_{L^p}.
$$

\qed

\section{Proof of Theorem \ref{thm:H} (II) }

We first prove \eqref{short variation singular} by a key lemma, its proof will be postponed until the end of the section.

\begin{lemma}\label{S2k} For $t\in[1,2)$ and $j\in\mathbb{Z}$, we define $\nu_{j,t}$ as
$$\nu_{j,t}(x)=\frac{\Omega(x')}{|x|^{n+1}}\chi_{\{2^jt\leq|x|\leq2^{j+1}\}}(x).$$
Denote $T_{j,t}$ by $T_{j,t}f(x)=\nu_{j,t}\ast f(x)$.  For $k\in\mathbb Z$, denote by $$S_{2,k}(\mathscr Cf)(x)=\Big(\sum_{j\in\mathbb{Z}}\|\{[b,T_{j,t}\Delta_{k-j}^2]
f(x)\}_{t\in[1,2)}\|_{V_2}^2\Big)^{\frac{1}{2}}.$$ For $b\in Lip({\Bbb R}^n),$ then the
following conclusions hold:\\

{\rm (i)}\quad  There exists  a constant $
\theta_1\in(0,1)$ such that
\begin{align}\label{S0}
\|S_{2,k}(\mathscr Cf) \|_{L^2}\le
C2^{-\theta_1 k}\|\Omega\|_{L^\infty(\mathbf S^{n-1})}\|
\nabla b\|_{L^\infty}\|f\|_{L^2};
\end{align}

{\rm (ii)}\quad   If $\Omega(x')$ satisfies \eqref{mean zero}, there exists  a constant $
\theta_2\in(0,1)$ such that
\begin{align}\label{Ss1}
\|S_{2,k}(\mathscr Cf) \|_{L^2}\le
C2^{\theta_2 k}\|\Omega\|_{L^1(\mathbf S^{n-1})}\|
\nabla b\|_{L^\infty}\|f\|_{L^2};\end{align}

{\rm (iii)}\quad For $1<p<\infty,$
\begin{align}\label{S2}
\|S_{2,k}(\mathscr Cf) \|_{L^p}\le C\|\Omega\|_{L^1(\mathbf S^{n-1})} \|
\nabla b\|_{L^\infty}\|f\|_{{L}^p}.
\end{align}
The constants $C's$ in \eqref{S0},\,\eqref{Ss1} and  \eqref{S2} are
independent of $k$.
\end{lemma}

Lemma \ref{S2k} will be proved at the end of this section. Let us
now finish the proof of \eqref{short variation singular} by using Lemma \ref{S2k}. For $t\in[1,2)$, let $\nu_{j,t}$ and $T_{j,t}$
be defined in the same way as  in Lemma \ref{S2k}.
 Observe that
$V_{2,j}(\mathscr Cf)(x)$ is just the strong  $2$-variation function of the family
$\{[b,T_{j,t}] f(x)\}_{t\in[1,2)}$, hence using $\sum_{l\in \Bbb Z}\Delta_l^2=\mathcal{I}$, we get
\begin{align*}
S_{2}(\mathscr Cf)(x)&=\Big(\sum_{j\in\mathbb{Z}}|V_{2,j}(\mathscr Cf)(x)|^2\Big)^{\frac{1}{2}}
=\Big(\sum_{j\in\mathbb{Z}}\|\{[b,T_{j,t}] f(x)\}_{t\in[1,2)}\|_{V_2}^2\Big)^{\frac{1}{2}}\\
&\leq\sum_{k\in\mathbb{Z}}\Big(\sum_{j\in\mathbb{Z}}\|\{[b,T_{j,t}\Delta_{k-j}^2]
f(x)\}_{t\in[1,2)}\|_{V_2}^2\Big)^{\frac{1}{2}}\\&:=
\sum_{k\in\mathbb{Z}}S_{2,k}(\mathscr Cf)(x)\\&=
\sum_{k<0}S_{2,k}(\mathscr Cf)(x)+\sum_{k\ge 0}S_{2,k}(\mathscr Cf)(x).
\end{align*}
Interpolating between \eqref{Ss1} and  \eqref{S2}, we can get for some  constant $\theta_3\in (0,1)$ and $1<p<\infty$,
\begin{align*}
 \|S_{2,k}(\mathscr Cf)\|_{L^p}\le C\|\Omega\|_{L^1(\mathbf S^{n-1})}2^{\theta_3k} \|
\nabla b\|_{L^\infty}\|f\|_{{L}^p},\,\,\hbox{for}\,\,\, k<0.
\end{align*}
Then for $1<p<\infty,$\begin{align*}
 \sum_{k<0}\|S_{2,k}(\mathscr Cf)\|_{L^p}&\le C\|\Omega\|_{L^1(\mathbf S^{n-1})}\|
\nabla b\|_{L^\infty}\|f\|_{{L}^p}\sum_{k<0}2^{\theta_3k}\\&\le \nonumber C\|\Omega\|_{L^1(\mathbf S^{n-1})}\|
\nabla b\|_{L^\infty}\|f\|_{{L}^p}.
\end{align*}  Decompose $\Omega(y')=\sum_{m\ge 0}\Omega_m(y')$ as in Section 3.
For $m\ge 0$, set $$\nu_{j,t,m}(x)=\frac{\Omega_m(x)}{|x|^{n+1}}\chi_{\{2^jt\le |x|< 2^{j+1}\}}(x),$$
$T_{j,t,m}$ is defined as $T_{j,t}$ by replacing $\nu_{j,t}$ by $\nu_{j,t,m}$.
\begin{align}\nonumber
\sum_{k\ge 0}S_{2,k}(\mathscr Cf)(x)&\le\dsum_{m\ge 0 }\dsum_{k\ge 0}\Big(\sum_{j\in\mathbb{Z}}\|\{[b,T_{j,t,m}\Delta_{k-j}^2]
f(x)\}_{t\in[1,2)}\|_{V_2}^2\Big)^{\frac{1}{2}}\\&\nonumber:=\dsum_{m\ge 0 }\dsum_{k\ge 0}
S_{2,k,m}(\mathscr Cf)(x).
\end{align}
Interpolating between \eqref{S0} and  \eqref{S2}, we can get for some constant $\theta_4\in (0,1/2)$ and $1<p<\infty$,
\begin{align}\label{S3}
 \|S_{2,k,m}(\mathscr Cf)\|_{L^p}\le C\|\Omega_m\|_{L^\infty(\mathbf S^{n-1})}2^{-\theta_4k} \|
\nabla b\|_{L^\infty}\|f\|_{{L}^p},\,\,\hbox{for}\,\,\, k\ge 0,\ \ m\ge0.
\end{align}
 Taking a large positive integer $N$, such that
$N>2\theta_4^{-1}.$ Then for $1<p<\infty,$ \begin{align*}\dsum_{k\ge 0}\|S_{2,k}(\mathscr Cf)\|_{L^p}\le\dsum_{m\ge 0 }\dsum_{k>Nm}\|S_{2,k,m}(\mathscr Cf)\|_{L^p}+\sum_{m\ge 0}\dsum_{0\le k\le Nm}\|S_{2,k,m}(\mathscr Cf)\|_{L^p}:=J_1+J_2.\end{align*}
 For $J_{1}$, using \eqref{S3}, we get for $1<p<\infty$,
\begin{align*}
J_{1}
&\leq C\|
\nabla b\|_{L^\infty}\dsum_{m\ge  0 }2^m\dsum_{k>Nm}2^{-\theta_4k}\|f\|_{L^p}\\
&\le C\|
\nabla b\|_{L^\infty}\|f\|_{L^p}.
\end{align*}
By \eqref{S2}, we get for $1< p<\infty$,
\begin{align*}J_{2}&\le
C\|
\nabla b\|_{L^\infty}\dsum_{m\ge  0}\dsum_{0\le k\le Nm}\|\Omega_m\|_{L^1(\mathbf S^{n-1})} \|f\|_{L^p}\\
&\le C\|
\nabla b\|_{L^\infty}\dsum_{m\ge  0}\dsum_{0\le k\le Nm}2^m
\sigma(E_{m}) \|f\|_{L^p}\\
&\le C\|
\nabla b\|_{L^\infty}\dsum_{m\ge
0}m2^m \sigma(E_{m}) \|f\|_{L^p}\\
&\le C\|\Omega\|_{L\log^+\!\!L(\mathbf S^{n-1})}\|
\nabla b\|_{L^\infty}\|f\|_{L^p}.
\end{align*}
Combining with the estimates of $J_1$ and $J_2,$ we get for $1<p<\infty$
$$\begin{array}{cl}\dsum_{k\ge 0}\|S_{2,k}(\mathscr Cf)\|_{L^p}\le
C(1+\|\Omega\|_{L\log^+\!\!L(\mathbf S^{n-1})})\|
\nabla b\|_{L^\infty}\|f\|_{L^p}.
\end{array}$$
We therefore finish the proof of \eqref{short variation singular}.

{\emph{\textbf{Proof of Lemma \ref{S2k}}.}}
To deal with \eqref{S0} and \eqref{Ss1}, we borrow the fact $\|\mathfrak{a}\|_{V_2}\le\|\mathfrak{a}\|_{L^2}^{1/2}\|\mathfrak{a}'\|_{L^2}^{1/2}$, where $\mathfrak{a}'=\{\frac{d}{dt}a_t:t\in\mathbb R\}$.
It is a special case of (39) in \cite{JSW08}.  Then,
\begin{align*}
[S_{2,k}(\mathscr Cf)(x)]^2&\leq\sum_{j\in\mathbb{Z}}\bigg(\dint_{1}^{2}|[b,T_{j,t}
\Delta_{k-j}^2]f(x)|^2\frac{dt}{t}\bigg)^{1/2}\bigg(\dint_{1}^{2}|\frac{d}{dt}[b, T_{j,t}
\Delta_{k-j}^2]f(x)|^2\frac{dt}{t}\bigg)^{1/2}.
\end{align*}
By the Cauchy-Schwarz inequality, we have
\begin{align*}\nonumber
\big\|S_{2,k}(\mathscr Cf)\big\|^2_{L^2}&\le\bigg\|\bigg(\dint_{1}^{2}\sum_{j\in\mathbb{Z}}|[b,T_{j,t}
\Delta_{k-j}^2]f(x)|^2\frac{dt}{t}\bigg)^{1/2}\bigg\|_{L^2}\bigg\|\bigg(\dint_{1}^{2}\sum_{j\in\mathbb{Z}}|\frac{d}{dt}[b, T_{j,t}
\Delta_{k-j}^2]f(x)|^2\frac{dt}{t}\bigg)^{1/2}\bigg\|_{L^2}\\&:=
\|I_{1,k}f\|_{L^2}\cdot\|I_{2,k}f\|_{L^2}.
\end{align*}
We estimate $\|I_{1,k}f\|_{L^2}$ and $\|I_{2,k}f\|_{L^2}$, respectively.
To estimate $\|I_{1,k}f\|_{L^2}$, we need the following
estimates
\begin{align}\label{nu1}
|\widehat{\nu_{j,t}}(\xi)|\leq C2^{-j}\|\Omega\|_{L^\infty(\mathbf S^{n-1})}|2^j\xi|^{-\gamma},\,\,\,\gamma\in(0,1)
       \end{align} and
for any fixed multi-index $\eta$ with $|\eta|\le 2,$\begin{align}\label{nu2}
|\partial^\eta\widehat{\nu_{j,t}}(\xi)|&\le C2^{j(|\eta|-1)}\|\Omega\|_{L^1(\mathbf S^{n-1})}
\end{align} uniformly in $t\in[1,2)$.
If $\Omega$ satisfies \eqref{mean zero}, then \begin{align}\label{nu3}
|\widehat{\nu_{j,t}}(\xi)|\leq C2^{-j}\|\Omega\|_{L^1(\mathbf S^{n-1})}
       |2^j\xi|^2
\end{align} and for any fixed multi-index $\eta$ with $|\eta|\le 2,$
 \begin{align}\label{nu4}
|\partial^\eta\widehat{\nu_{j,t}}(\xi)|&\le C2^{j(|\eta|-1)}\|\Omega\|_{L^1(\mathbf S^{n-1})}|2^j\xi|^{2-|\eta|}
\end{align}
uniformly in $t\in[1,2)$.
Set $M_{j,t}(\xi)=\widehat{\nu_{j,t}}(\xi),$ $M_{j,t}^k(\xi)=M_{j,t}(\xi)\varphi(2^{j-k}\xi)$. We define multipliers $T_{j,t}^k$ and $\widetilde{T}_{j,t}^k$, respectively by
$$\widehat{T_{j,t}
^kf}(\xi)=M_{j,t}^k(\xi)\widehat{f}(\xi)\ \ \text{and}\ \ \widehat{\widetilde{T}_{j,t}^k} f(\xi)= M_{j,t}^k(2^{-j}\xi)\widehat{f}(\xi).$$
We use \eqref{nu1}-\eqref{nu2} to get for $k\ge 0,$
\begin{align}\label{Mj1}\|M_{j,t}^k(2^{-j}\cdot)\|_{L^\infty}\le C2^{-j}2^{-\gamma k}\|\Omega\|_{L^\infty(\mathbf S^{n-1})},\,\,\,\,\|\partial^\alpha[M_{j,t}^k(2^{-j}\cdot)]\|_{L^\infty}\le C2^{-j}\|\Omega\|_{L^1(\mathbf S^{n-1})},\end{align}  where   $\alpha$ is  a  multi-index with $|\alpha|= 2.$
 Using \eqref{nu3}-\eqref{nu4}, we get for $k< 0,$
\begin{align}\label{Mj2}\|M_{j,t}^k(2^{-j}\cdot)\|_{L^\infty}\le C2^{-j} 2^{2k}\|\Omega\|_{L^1(\mathbf S^{n-1})},\,\,\,\,\,\|\partial^\alpha[M_{j,t}^k(2^{-j}\cdot)]\|_{L^\infty}\le C2^{-j}\|\Omega\|_{L^1(\mathbf S^{n-1})},\end{align} where   $\alpha$ is  a  multi-index with $|\alpha|= 2.$
Via Lemma \ref{CD1} to \eqref{Mj1} and \eqref{Mj2} with $\delta=2^k,$ respectively states  that for any fixed $0<v<1,$
$$\|[b, \widetilde{T}_{j,t}^k] f\|_{L^2}\le
C2^{-j}\min\{2^{-\gamma v k}\|\Omega\|_{L^\infty(\mathbf S^{n-1})}, 2^{2vk}\|\Omega\|_{L^1(\mathbf S^{n-1})}\}\|
 b\|_{Lip} \|f\|_{L^2}.$$
Then by $\|b(2^j\cdot)\|_{Lip}=2^j\|b\|_{Lip}$ says that
\begin{align}\label{bt}
\|[b,T_{j,t}^k
] f\|_{L^2}\le
C\min\{2^{-\gamma v k}\|\Omega\|_{L^\infty(\mathbf S^{n-1})}, 2^{2vk}\|\Omega\|_{L^1(\mathbf S^{n-1})}\}\|
 b\|_{Lip} \|f\|_{L^2}.
\end{align}
 By Plancherel theorem, we also get
\begin{align}\label{btp}
\|T_{j,t}^k
 f\|_{L^2}\le
C2^{-j}\min\{2^{-\gamma  k}\|\Omega\|_{L^\infty(\mathbf S^{n-1})}, 2^{2k}\|\Omega\|_{L^1(\mathbf S^{n-1})}\}\|f\|_{L^2}.
\end{align}
 Write
$$
[b,T_{j,t}\Delta_{k-j}^2 ]f=[b,T_{j,t}^k\Delta_{k-j} ]f=
[b, T_{j,t}^k ]\Delta_{k-j}f+T_{j,t}^k [b, \Delta_{k-j}]f.$$
Then we get
\begin{align*}
\|I_{1,k}f\|_{L^2}\le \bigg(\dint_{1}^{2}\sum_{j\in\mathbb{Z}}\|[b,T_{j,t}^k
]\Delta_{k-j}f\|_{L^2}^2\frac{dt}{t}\bigg)^{1/2}+ \bigg(\dint_{1}^{2}\sum_{j\in\mathbb{Z}}\|T_{j,t}^k
[b,\Delta_{k-j}]f\|_{L^2}^2\frac{dt}{t}\bigg)^{1/2}.
\end{align*}
By  \eqref{bt} and Littlewood-Paley theory, we get
\begin{align}\label{K1}
& \bigg(\dint_{1}^{2}\sum_{j\in\mathbb{Z}}\|[b,T_{j,t}^k
]\Delta_{k-j}f\|_{L^2}^2\frac{dt}{t}\bigg)^{1/2}\\&\le \nonumber  C\min\{2^{-\gamma v k}\|\Omega\|_{L^\infty(\mathbf S^{n-1})}, 2^{2vk}\|\Omega\|_{L^1(\mathbf S^{n-1})}\}\|
\nabla b\|_{L^\infty} \bigg(\dint_{1}^{2}\sum_{j\in\mathbb{Z}}\|\Delta_{k-j}f\|_{L^2}^2\frac{dt}{t}\bigg)^{1/2}\\&\le \nonumber C\min\{2^{-\gamma v k}\|\Omega\|_{L^\infty(\mathbf S^{n-1})}, 2^{2vk}\|\Omega\|_{L^1(\mathbf S^{n-1})}\} \|
\nabla b\|_{L^\infty}\|f\|_{L^2}.
\end{align}
By \eqref{btp} and Lemma \ref{Deltacom}, we get
\begin{align}\label{K2}
 &\bigg(\dint_{1}^{2}\sum_{j\in\mathbb{Z}}\|T_{j,t}^k
[b,\Delta_{k-j}]f\|_{L^2}^2\frac{dt}{t}\bigg)^{1/2}\\&\le \nonumber  C\min\{2^{-(\gamma+1) k}\|\Omega\|_{L^\infty(\mathbf S^{n-1})}, 2^{k}\|\Omega\|_{L^1(\mathbf S^{n-1})}\} \bigg(\dint_{1}^{2}\sum_{j\in\mathbb{Z}}\|2^{k-j}[b,\Delta_{k-j}]f\|_{L^2}^2\frac{dt}{t}\bigg)^{1/2}\\&\le \nonumber C\min\{2^{-(\gamma+1) k}\|\Omega\|_{L^\infty(\mathbf S^{n-1})}, 2^{k}\|\Omega\|_{L^1(\mathbf S^{n-1})}\}\|
\nabla b\|_{L^\infty}\|f\|_{L^2}.\end{align}
 Combining the estimates of \eqref{K1} and \eqref{K2}, we get
\begin{align}\label{I1k}
\|I_{1,k}f\|_{L^2}\le C\min\{2^{-\gamma v k}\|\Omega\|_{L^\infty(\mathbf S^{n-1})}, 2^{vk}\|\Omega\|_{L^1(\mathbf S^{n-1})}\}\|
\nabla b\|_{L^\infty}\|f\|_{L^2}.
\end{align}
Next, we estimate $\|I_{2,k}f\|_{L^2}$. We write
\begin{align*}
&\dfrac{d}{dt}[b,T_{j,t}\Delta_{k-j}^2 ]f=
\dfrac{d}{dt}T_{j,t} [b, \Delta_{k-j}^2]f+[b, \dfrac{d}{dt}T_{j,t} ]\Delta_{k-j}^2f.
\end{align*}
Then we get
\begin{align*}
\|I_{2,k}f\|_{L^2}&\le\bigg(\dint_{1}^{2}\bigg\|\bigg(\dsum_{j\in \Bbb Z}| \dfrac{d}{dt}T_{j,t}
[b,\Delta_{k-j}^2]f|^2\bigg)^{1/2}\bigg\|_{L^2}^2\frac{dt}{t}\bigg)^{1/2}\\&+\bigg(\dint_{1}^{2}\bigg\|\bigg(\dsum_{j\in \Bbb Z}| [b,\dfrac{d}{dt}T_{j,t}]
\Delta_{k-j}^2f|^2\bigg)^{1/2}\bigg\|_{L^2}^2\frac{dt}{t}\bigg)^{1/2}\\&:=L_1+L_2.
\end{align*}
To estimate $L_i$
for $i=1,2$, respectively,
we need the following elementary fact
\begin{align}\label{Tjt}
\frac{d}{dt}T_{j,t}h(x)&=\dfrac{d}{dt}\bigg[\int_{2^jt<|y|\le 2^{j+1}}\dfrac{\Omega(y')}{|y|^{n+1}}h(x-y)dy\bigg]\\
&\nonumber=\dfrac{d}{dt}\bigg[\dint_{\mathbf S^{n-1}}\Omega(y')\dint_{2^jt}^{2^{j+1}}\frac{1}{r^2}h(x-ry')drd\sigma(y')\bigg]\\
&\nonumber=-\dfrac{1}{2^jt^2}\dint_{\mathbf S^{n-1}}\Omega(y')h(x-2^jty')d\sigma(y')\\
&\le\nonumber\dsum_{i=1}^n\dint_0^1\dint_{S^{n-1}}|\Omega(y')|
|\partial_i h(x+s2^jty')|t^{-1} d\sigma(y')\,ds.
\end{align}
For $t\in [1,2)$, it is easy to get
\begin{align}\label{Tjt1}
\|\frac{d}{dt}T_{j,t}h\|_{L^2}\le C\|\Omega\|_{L^1(\mathbf S^{n-1})}\|\nabla h\|_{L^2}.
\end{align}
On the other hand, we get
$$\begin{array}{cl}
[b,\frac{d}{dt}T_{j,t}]h(x)
&=-\dfrac{1}{2^jt^2}\dint_{\mathbf S^{n-1}}\Omega(y')(b(x)-b(x-2^jty'))h(x-2^jty')d\sigma(y')\\
&\le\|\nabla b\|_{L^\infty}\dint_{\mathbf S^{n-1}}|\Omega(y')|
|h(x-2^jty')|t^{-1} d\sigma(y').
\end{array}$$For $t\in [1,2)$, it is easy to get
\begin{align}\label{Tjt2}
\|[b,\frac{d}{dt}T_{j,t}]h\|_{L^2}\le C\|\Omega\|_{L^1(\mathbf S^{n-1})}\|\nabla b\|_{L^\infty}\| h\|_{L^2}.
\end{align}
We now estimate $L_1$.
Indeed, by  \eqref{Tjt1} and \eqref{Lp2},  we
have
\begin{align*}
L_1&\le C\|\Omega\|_{L^1(\mathbf S^{n-1})}\bigg(\dint_{1}^{2} \dsum_{j\in
\Bbb Z}\big\|\nabla[b,
\Delta_{k-j}^2]f
\big\|_{L^2}^2\frac{dt}{t}\bigg)^{1/2}\\
&\le C\|\Omega\|_{L^1(\mathbf S^{n-1})}\|\nabla b\|_{L^\infty}\|f\|_{L^2}.
\end{align*}
Similarly, by  \eqref{Tjt2} and  Littlewood-Paley theory, we
have
\begin{align*}
L_2&\le C\|\Omega\|_{L^1(\mathbf S^{n-1})}\|\nabla b\|_{L^\infty}\bigg(\dint_{1}^{2} \dsum_{j\in
\Bbb Z}\big\|
\Delta_{k-j}^2f
\big\|_{L^2}^2\frac{dt}{t}\bigg)^{1/2}\\
&\le C\|\Omega\|_{L^1(\mathbf S^{n-1})}\|\nabla b\|_{L^\infty}\|f\|_{L^2}.
\end{align*}
Combining the estimates of $L_1$ and $L_2$, we get
\begin{align*}
\|I_{2,k}f\|_{L^2}\le  C\|\Omega\|_{L^1(\mathbf S^{n-1})}\|\nabla b\|_{L^\infty}\|f\|_{{L}^2}.
\end{align*}
Then combined with the  estimate of \eqref{I1k}, we get for  $k\in \Bbb Z,$
\begin{align*}
\big\|S_{2,k}(\mathscr Cf)\big\|^2_{L^2}&\le C\min\{ 2^{-\gamma vk}\|\Omega\|_{L^\infty(\mathbf S^{n-1})}^2, 2^{vk}\|\Omega\|_{L^1(\mathbf S^{n-1})}^2\}\|
\nabla b\|_{L^\infty}^2 \|f\|_{L^2}^2.
\end{align*} This finishes the proof of \eqref{S0} and \eqref{Ss1}.

So to prove Lemma 7.1, it suffices to prove \eqref{S2}. Let $$
B=\Big\{\{a_{j,t}\}_{j\in \Bbb Z,\, t\in [1,2)}:\|{a_{j,t}}\|_B:=\big(\sum_{j\in \Bbb Z}\|{a_{j,t}}\|^2_{V_2}\big)^{1/2}<\infty\Big\}.
$$
Clearly, $(B,\|\cdot\|_B)$ is a Banach space.Then,
\begin{align*}
S_{2,k}(\mathscr Cf)(x)&=\Big(\dsum_{j\in\mathbb{Z}}\dsup_{\substack
{t_1<\cdots<t_N\\
[t_l,t_{l+1}]\subset[1,2)}}\dsum_{l=1}^{N-1}\big|[b,T_{j,t_l}\Delta_{k-j}^2]
f(x)-[b,T_{j,t_{l+1}}\Delta_{k-j}^2]
f(x)|^2\Big)^{\frac{1}{2}}\\
&=\Big(\dsum_{j\in\mathbb{Z}}\dsup_{\substack
{t_1<\cdots<t_N\\
[t_l,t_{l+1}]\subset[1,2)}}\dsum_{l=1}^{N-1}\big|[b,T_{j,t_l,t_{l+1}}\Delta_{k-j}^2]
f(x)|^2\Big)^{\frac{1}{2}},
\end{align*} where $$T_{j,t_l,t_{l+1}}f(x)=\dint_{2^jt_l<|y|\le 2^jt_{l+1}}f(x-y)\frac{\Omega(y)}{|y|^{n+1}}dy\,\,\, \hbox{and}\,\,\,\,[t_l,t_{l+1}]\subset[1,2).$$
Then by $$[b,T_{j,t_l,t_{l+1}}\Delta_{k-j}^2]f=[b,T_{j,t_l,t_{l+1}}]\Delta_{k-j}^2
f+T_{j,t_l,t_{l+1}}[b,\Delta_{k-j}^2]
f,$$ we get
\begin{align*}
S_{2,k}(\mathscr Cf)(x)&\le \Big(\dsum_{j\in\mathbb{Z}}\dsup_{\substack
{t_1<\cdots<t_N\\
[t_l,t_{l+1}]\subset[1,2)}}\dsum_{l=1}^{N-1}\big|[b,T_{j,t_l,t_{l+1}}]\Delta_{k-j}^2
f(x)|^2\Big)^{\frac{1}{2}}\\&+\Big(\dsum_{j\in\mathbb{Z}}\dsup_{\substack
{t_1<\cdots<t_N\\
[t_l,t_{l+1}]\subset[1,2)}}\dsum_{l=1}^{N-1}\big|T_{j,t_l,t_{l+1}}[b,\Delta_{k-j}^2]
f(x)|^2\Big)^{\frac{1}{2}}.
\end{align*}
By the mean value zero property of $\Omega$,  we have
$$\begin{array}{cl}T_{j,t_l,t_{l+1}}f(x)&=\dint_{2^jt_l}^{ 2^jt_{l+1}}\dint_{S^{n-1}}f(x-ry')\Omega(y')d\sigma(y')\frac{dr}{r^2}\\&=\dint_{2^jt_l}^{ 2^jt_{l+1}}\dint_{S^{n-1}}\Omega(y')
\Big(f(x-ry')-f(x)\Big)d\sigma(y')\frac{dr}{r^2}\\&=\dsum_{|\beta|=1}\dint_{2^jt_l}^{ 2^jt_{l+1}}\dint_0^1\dint_{S^{n-1}}\Omega(y')
D^\beta f(x+sry')(ry')^\beta d\sigma(y')\,ds\frac{dr}{r^2}.\end{array}$$
For $[t_l,t_{l+1}]\subset[1,2)$, let $${\widetilde{T}}_{j,t_l,t_{l+1}}f(x)=\dint_{2^jt_l}^{ 2^jt_{l+1}}\dint_0^1\dint_{S^{n-1}}|\Omega(y')|
|\nabla f(x+sry')| d\sigma(y')\,ds\frac{dr}{r}$$ and $${T}^*_{j,t_l,t_{l+1}}f(x)=\dint_{2^jt_l<|y|\le 2^jt_{l+1}}|f(x-y)|\frac{|\Omega(y)|}{|y|^n}dy.$$
Then,
\begin{align*}
S_{2,k}(\mathscr Cf)(x)&\le C\|\nabla b\|_{L^\infty}\Big(\dsum_{j\in\mathbb{Z}}\dsup_{\substack
{t_1<\cdots<t_N\\
[t_l,t_{l+1}]\subset[1,2)}}\dsum_{l=1}^{N-1}\big|{{T}}^*_{j,t_l,t_{l+1}}\Delta_{k-j}^2
f(x)|^2\Big)^{\frac{1}{2}}\\&+C\Big(\dsum_{j\in\mathbb{Z}}\dsup_{\substack
{t_1<\cdots<t_N\\
[t_l,t_{l+1}]\subset[1,2)}}\dsum_{l=1}^{N-1}\big|\widetilde{T}_{j,t_l,t_{l+1}}[b,\Delta_{k-j}^2]
f(x)|^2\Big)^{\frac{1}{2}}\\&=C\|\nabla b\|_{L^\infty} \Big(\dsum_{j\in\mathbb{Z}}\dsup_{\substack
{t_1<t_N\\
[t_1,t_{N}]\subset[1,2)}}\big|{T}^*_{j,t_1,t_{N}}(\Delta_{k-j}^2
f)(x)|^2\Big)^{\frac{1}{2}}\\&+C\Big(\dsum_{j\in\mathbb{Z}}\dsup_{\substack
{t_1<t_N\\
[t_l,t_{N}]\subset[1,2)}}\big|\widetilde{T}_{j,t_1,t_{N}}([b,\Delta_{k-j}^2]
f)(x)|^2\Big)^{\frac{1}{2}}.
\end{align*}
 Therefore, we get
 $$
 S_{2,k}(\mathscr Cf)(x)
 \le C\|\nabla b\|_{L^\infty}\Big(\dsum_{j\in\mathbb{Z}}\big|M_{\Omega}(\Delta_{k-j}^2
f)(x)|^2\Big)^{\frac{1}{2}}+C\Big(\dsum_{j\in\mathbb{Z}}\big|M_{\Omega}(\nabla[b,\Delta_{k-j}^2]
f)(x)|^2\Big)^{\frac{1}{2}}.$$
To proceeding with the estimates, we need the following inequality, for $1<p<\infty$
$$
\bigg\|
 \bigg(\dsum_{j\in \Bbb Z}
 |M_{\Omega} f_j|^2\bigg)^{1/2}\bigg\|_{L^p}\le C\|\Omega\|_{L^1(\mathbf{S}^{n-1})}\bigg\|
 \bigg(\dsum_{j\in \Bbb Z}
 | f_j|^2\bigg)^{1/2}\bigg\|_{L^p},
$$
which was established in  \cite{CD}.
Then by  Littlewood-Paley theory  and \eqref{Lp2}, we
have for $1<p<\infty$
\begin{align*}
&\|S_{2,k}(\mathscr Cf)\|_{L^p}\\
\le& C\|\nabla b\|_{L^\infty}\|\Omega\|_{L^1(\mathbf{S}^{n-1})}\bigg\|\Big(\dsum_{j\in\mathbb{Z}}|\Delta_{k-j}^2
f|^2\Big)^{\frac{1}{2}}\bigg\|_{L^p}+C\|\Omega\|_{L^1(\mathbf{S}^{n-1})}\bigg\|\Big(\dsum_{j\in\mathbb{Z}}\big|\nabla[b,\Delta_{k-j}^2]
f|^2\Big)^{\frac{1}{2}}\bigg\|_{L^p}\\
\le& C\|\nabla b\|_{L^\infty}\|\Omega\|_{L^1(\mathbf{S}^{n-1})}\|f\|_{L^p}.\end{align*}
This gives \eqref{S2}. Therefore, we complete the proof Lemma \ref{S2k}. \qed

\section{Proof of  Corollary \ref{thm:H2}}  For $\varepsilon>0,$ write
$$\begin{array}{cl}
[b,T_\varepsilon]\nabla f(x)=-T_\varepsilon[b, \nabla]f(x)+[b, \nabla T_\varepsilon]f(x).\end{array}$$
 For the first
term, since $[b, \nabla]f=-(\nabla b)f,$ then by Theorem 1.2 in \cite{DHL},  for $1<p<\infty,$
we have \begin{align}\label{T}\|\lambda\sqrt{N_\lambda(\{T_\varepsilon[b, \nabla]f\}_{\varepsilon>0})}\|_{L^p}&\le C\|K\|_{L(\log^+L)^2(\mathbf S^{n-1})}\|(\nabla b) f\|_{L^p}\\&\le \nonumber C\|K\|_{L(\log^+L)^2(\mathbf S^{n-1})}\|\nabla b\|_{L^\infty}\|f\|_{L^p}.\end{align}
For the second term,  it is easy to verify that  $\nabla K(x)$ is homogeneous of degree $-n-1$ and \begin{align*}(x_k\nabla K(x))^{\wedge}(\xi)&=i\xi_k\widehat{\nabla K}(\xi)=i\dfrac{\partial}{\partial \xi_k}(i\xi_1\widehat{K}(\xi),\dots,i\xi_n\widehat{K}(\xi)).\end{align*}
Moreover, if $j=k,$ then  $
\dfrac{\partial}{\partial \xi_k}(\xi_j\widehat{K})(\xi)=\widehat{ K}(\xi)+\xi_j\dfrac{\partial\widehat{K}(\xi)}{\partial \xi_k}.$ If $j\neq k,$ then   $
\dfrac{\partial}{\partial \xi_k}(\xi_j\widehat{K})(\xi)=\xi_j\dfrac{\partial\widehat{K}(\xi)}{\partial \xi_k}.$
So we get for $k=1,\dots,n,$  $(x_k\nabla K(x))^{\wedge}(0)=0.$
Additionally, $\widehat{ \nabla K}(\xi)=i \xi\widehat{ K}(\xi)$, then $\widehat{\nabla K}(0)=0.$ This says that $$\int_{\mathbf S^{n-1}}(x_k')^N\nabla K(x')\,d\sigma(x')=0$$ for any $k=1,\dots,n$ and $N=0,1.$
Since $|\nabla K(x')|\in L(\log^+L)^2(\mathbf S^{n-1}),$ then by  Theorem \ref{thm:H},  the family of the  operators \begin{align*}\{[b, \nabla T_\varepsilon] f(x)\}_{\varepsilon>0}= \bigg\{\dint_{|x-y|>\varepsilon}\nabla K(x-y)(b(x)-b(y) f(y)\,dy\bigg\}_{\varepsilon>0}\end{align*}  satisfies  \begin{align}\label{C}\|\lambda\sqrt{N_\lambda(\{[b, \nabla T_\varepsilon]f\}_{\varepsilon>0} )}\|_{L^p}\le
C\|\nabla K\|_{L(\log^+L)^2(\mathbf S^{n-1})}\|\nabla b\|_{L^\infty}\|f\|_{L^p},\ 1<p<\infty.\end{align}
Combining the estimates of \eqref{T} and \eqref{C}, we get
 \begin{align}\nonumber\|\lambda\sqrt{N_\lambda(\mathcal{T}_b \nabla f )}\|_{L^p}&\le
C(\|K\|_{L(\log^+L)^2(\mathbf S^{n-1})}+\|\nabla K\|_{L(\log^+L)^2(\mathbf S^{n-1})})\|\nabla b\|_{L^\infty}\|f\|_{L^p},\ 1<p<\infty.\end{align}
 For $\{\nabla[b,T_\varepsilon] f\}_{\varepsilon>0},$ we have
$$\begin{array}{cl}
\nabla[b,T_\varepsilon] f(x)&=-[b, \nabla]T_\varepsilon f(x)+[b, \nabla T_\varepsilon]f(x)=-(\nabla b)(x)T_\varepsilon f(x)+[b, \nabla T_\varepsilon]f(x).\end{array}$$
Similarly to the proof of Theorem 1.2 in \cite{DHL}, we can get for $1<p<\infty,$
\begin{align}\|\lambda\sqrt{N_\lambda(\{(\nabla b)T_\varepsilon f\}_{\varepsilon>0})}\|_{L^p}&\le \nonumber C\|K\|_{L(\log^+L)^2(\mathbf S^{n-1})}\|\nabla b\|_{L^\infty}\|f\|_{L^p}.\end{align} Then by \eqref{C}, we get for $1<p<\infty,$
\begin{align}\nonumber\|\lambda\sqrt{N_\lambda( \nabla \mathcal{T}_bf )}\|_{L^p}&\le
C(\|K\|_{L(\log^+L)^2(\mathbf S^{n-1})}+\|\nabla K\|_{L(\log^+L)^2(\mathbf S^{n-1})})\|\nabla b\|_{L^\infty}\|f\|_{L^p}.\end{align} \qed\bibliographystyle{amsplain}

\end{document}